\documentclass[12pt]{amsart}
\usepackage{amssymb,amsmath,amsthm}
\usepackage[all,cmtip]{xy}
\usepackage{hyperref}
\usepackage[msc-links, backrefs]{amsrefs}
\hypersetup{
  colorlinks	= true,	
  urlcolor		= blue,	
  linkcolor		= blue,	
  citecolor		= red		
}
\calclayout
\newcommand{\beq}{\begin{small} \begin{equation}}
\newcommand{\eeq}{\end{equation} \end{small}}
\newcommand{\beqn}{\begin{small} \begin{equation*}}
\newcommand{\eeqn}{\end{equation*} \end{small}}
\newtheorem{theorem}{Theorem}[section]
\newtheorem{proposition}[theorem]{Proposition}
\newtheorem{corollary}[theorem]{Corollary}
\newtheorem{lemma}[theorem]{Lemma}
\newtheorem{remark}[theorem]{Remark}
\newtheorem{definition}[theorem]{Definition}
\numberwithin{theorem}{section} \numberwithin{equation}{section}
\newcommand{\hpg}[5]{{}_{#1}F_{#2}\! \left(\left.{\genfrac{}{}{0pt}{}{#3}{#4}}\right| #5 \right) }
\begin{document}
\title[Kummer sandwiches and Greene-Plesser construction]{Kummer sandwiches and Greene-Plesser construction}
\author{Noah Braeger}
\author{Andreas Malmendier}
\author{Yih Sung}
\email{noah.braeger@usu.edu, andreas.malmendier@usu.edu, yih.sung@usu.edu}
\address{Department of Mathematics \& Statistics, Utah State University, Logan, UT 84322}
\begin{abstract}
In the context of K3 mirror symmetry, the Greene-Plesser orbifolding method constructs a family of K3 surfaces, the mirror of quartic hypersurfaces in $\mathbb{P}^3$, starting from a special one-parameter family of K3 varieties known as the quartic Dwork pencil. We show that certain K3 double covers obtained from the three-parameter family of quartic Kummer surfaces associated with a principally polarized abelian surface generalize the relation of the Dwork pencil and the quartic mirror family. Moreover, for the three-parameter family we compute a formula for the rational point-count of its generic member and derive its transformation behavior with respect to $(2,2)$-isogenies of the underlying abelian surface.
\end{abstract}
\keywords{K3 surfaces, Kummer sandwich theorem, Greene-Plesser orbifolding method}
\subjclass[2010]{14J28, 14J33, 33C65}
\maketitle
\section{Introduction}
In the theory of algebraic curves, Manin's celebrated unity theorem \cite{MR1946768}*{Sec.~2.12}, provides a connection between certain period integrals for families of algebraic curves and the number of rational points over finite fields $\mathbb{F}_p$ on them. The correspondence is established using the Gauss-Manin connection and the holomorphic solution for the resulting Picard-Fuchs equation. In \cite{MR3613974} the third author explores the connection between the rational point-count and periods of a one parameter family of curves associated with with triangle groups; in \cite{MR3992148} the second and the third author demonstrated how this principle is generalized to include two-parameter families of Kummer surfaces. The goal of this article is to extend our technique further: we will focus on certain K3 double covers obtained from the three-parameter family of quartic Kummer surfaces associated with a generic principally polarized abelian surface. We compute the rational point-count for the generic member of this family and determine its transformation with respect to $(2,2)$-isogenies of the underlying abelian surface. 
\par The underlying geometric techniques in this article are motivated by mirror symmetry: in their seminal work of Candelas et.~al \cite{MR2019149} arithmetic properties of the periods of the famous Dwork pencil of quintic threefolds were derived, and it was shown that for an understanding of a quantum version of the congruence zeta function arithmetic properties of the periods are crucial. One feature of the Greene-Plesser orbifolding construction is that the so-called mirror map can be computed explicitly; see \cites{MR1101784, MR1416334}. 
\par In this article, we will focus on K3 mirror symmetry. Here, the Greene-Plesser orbifolding method constructs a family of K3 surfaces, the mirror of general quartics in $\mathbb{P}^3$, starting from the Dwork pencil. As we will demonstrate,  the famous period computation of Narumiya and Shiga~\cite{MR1877764} can then be traced back to the existence of suitable Kummer sandwich theorems for the quartic mirror family. After we establish the connection between the Greene-Plesser orbifolding method and these Kummer sandwich theorems, we will generalize these theorems to the three-parameter family of quartic Kummer surfaces, and then compute the rational point-counting function of a generic member of the three-parameter family generalizing the quartic mirror.
\par Concretely, the Dwork pencil is the one-parameter family of deformed Fermat hypersurfaces in the projective space $\mathbb{P}^{n}=\mathbb{P}(X_0,\dots,X_n)$ given by
\beq
\label{Fermat}
 X_0^{n+1} + X_1^{n+1} + \dots + X_{n}^{n+1} + (n+1) \, \lambda \, X_0 X_1 \cdots X_{n} = 0 \,.
\eeq
It is known that for each integer $n\in \mathbb{N}$ the smooth resolution of Equation~\eqref{Fermat} constitutes a family of $(n-1)$-dimensional Calabi-Yau hypersurfaces $\mathcal{X}_{\lambda}$. For $n=4$ Equation~\eqref{Fermat} is the famous quintic family of Candelas et al.~\cite{MR1101784}. For the family~\eqref{Fermat} a discrete group of symmetries is identified as follows: it is generated by the action $(X_0,X_j) \mapsto (\zeta_{n+1}^n X_0, \zeta_{n+1} X_j)$ for $1 \le j \le n$ with $\zeta_{n+1}=\exp{(\frac{2\pi i}{n+1})}$.  Since the product of all generators multiplies the homogeneous coordinates by a common phase,  the symmetry group is $G_{n-1}=(\mathbb{Z}/(n+1)\, \mathbb{Z})^{n-1}$. The new parameter and affine variables
\begin{small}\begin{gather*}
 \mu=\frac{(-1)^{n+1}}{\lambda^{n+1}} \,, \;  x_1 = \frac{X_1^n}{(n+1)\, X_0 \cdot X_2 \cdots X_n \, \lambda}  \,, \;
 x_2 =\frac{X_2^n}{(n+1)\, X_0 \cdot X_1 \cdot X_3 \cdots X_n \, \lambda} ,  \; \dots \; ,
 \end{gather*}
 \end{small}%
are invariant under the action of $G_{n-1}$. Hence, they descend to coordinates on the orbifold quotient $\mathcal{X}_{\lambda}/G_{n-1}$. A birational model for $\mathcal{X}_{\lambda}/G_{n-1}$ is then given in these new affine variables $x_1, \dots, x_n$ using the remaining relation between them, namely
\beq
\label{mirror_family}
 f_n(x_1,\dots, x_n, t) = x_1 \cdots x_n \, \Big( x_1 + \dots + x_n + 1 \Big) +  \frac{(-1)^{n+1} \, \mu}{(n+1)^{n+1}} =0 \,.
\eeq
\par It was proved in ~\cite{MR1416334} that a family of Calabi-Yau hypersurfaces $\mathcal{Y}_{\mu}$ of degree $(n+1)$ in $\mathbb{P}^n$ can be obtained from Equation~\eqref{mirror_family} after the resolution of its singularities. This is known as the Greene-Plesser orbifolding construction; see \cites{MR1059831, MR1113571}. In fact, we have the following general proposition \cite{MR1677117}*{Prop.~4.2.3}:
\begin{proposition}
Let $\mathcal{X}$ be a Calabi-Yau variety and $G$ be a discrete group of symmetries on $\mathcal{X}$. Then the smooth resolution of the orbifold $\mathcal{X}/G$ as well as its deformations are again Calabi-Yau varieties.
\end{proposition}
The subspace of the cohomology $H^{n-1}(\mathcal{X}_{\lambda},\mathbb{Q})$ which is invariant under the action of $G_{n-1}$ or, equivalently, the cohomology $H^{n-1}(\mathcal{Y}_{\mu},\mathbb{Q})$ has dimension $n$ and Hodge numbers $(1, \dots, 1)$. Thus, the family $\mathcal{Y}_{\mu}$ is the \emph{mirror family} of the hypersurfaces $\mathcal{Z}$ in $\mathbb{P}^n$ of degree $(n+1)$ and co-dimension one in $\mathbb{P}^n$, in the sense of mirror symmetry motivated by string theory \cite{MR1101784}. There is also a formulation of this construction due to Batyrev which works for any family of Calabi-Yau hypersurfaces in toric varieties, constructed from reflexive polytopes, such that a dual pair of reflexive polytopes gives rise to the mirror pair of Calabi-Yau hypersurfaces \cites{MR1416334, MR1860046}.
\par We prove in Section~\ref{sec:K3_19} that the Greene-Plesser orbifolding construction for K3 surfaces factors through a Kummer sandwich that can be established for the family $\mathcal{Y}_\mu$, implying that each K3 surface is dominated and dominates a Kummer surface of Picard rank 19; see Theorem~\ref{prop4}. In particular, there are rational maps
\beq
 \mathcal{X}_\lambda \dashrightarrow \operatorname{Kum}(\mathcal{E} \times \mathcal{E}') \dashrightarrow \mathcal{Y}_\mu 
 \dashrightarrow \operatorname{Kum}(\mathcal{E} \times \mathcal{E}')  \,,
\eeq 
where $\mathcal{E}$ and $\mathcal{E}'$ are the two-isogeneous elliptic curves. This fact implies that the quartic mirror family is the rational cover of a twisted Legendre pencil $\mathcal{Y}'_\mu$ whose period mapping can be computed easily. In this way, the famous period computation of Narumiya and Shiga~\cite{MR1877764} is seen to be a consequence of the existence of suitable Kummer sandwich theorems for the quartic mirror family.
\par  In Section~\ref{sec:K3_17} we will show that certain K3 double covers obtained from the three-parameter family of quartic Kummer surfaces generalize many features present for the quartic mirror family. Concretely, the general Kummer quartic
\begin{small}
\begin{gather}
\label{Goepel-Quartic_intro}
  0  = X_0^4+X_1^4+X_2^4+X_3^4  + 2  D  X_0 X_1 X_2 X_3\\
     - A  \big(X_0^2 X_1^2+X_2^2 X_3^2\big)  - B \big(X_0^2 X_2^2+X_1^2 X_3^2\big) - C  \big(X_0^2 X_3^2+X_1^2 X_2^2\big)   \,, 
\end{gather}
\end{small}%
where $A, B, C, D \in \mathbb{C}$ and $D^2 = A^2 + B^2 + C^2 + ABC - 4$, is the multi-parameter generalization of the quartic Dwork pencil in Equation~\eqref{Fermat} with a discrete group of symmetries broken from $(\mathbb{Z}/4 \mathbb{Z})^2$ to $G=(\mathbb{Z}/2 \mathbb{Z})^2$. The general Kummer quartic is the Kummer surface $\operatorname{Kum}(\mathcal{A})$ associated with a principally polarized abelian surface $\mathcal{A}$. The minimal resolution of $\operatorname{Kum}(\mathcal{A})/G$ is again a Kummer surface, namely $\operatorname{Kum}(\mathcal{A}')$ associated with the $(2,2)$-isogeneous abelian surface $\mathcal{A}'$. We also find that the aforementioned Kummer sandwich theorem generalizes to Picard rank $17$; see Theorem~\ref{prop6}. In particular, there are rational maps of degree two such that
\beq
\label{eqn:sandwich_intro}
\operatorname{Kum}(\mathcal{A}) \dashrightarrow \mathcal{Y}  \dashrightarrow \operatorname{Kum}(\mathcal{A}) \quad \text{and} \quad
\operatorname{Kum}(\mathcal{A}) \dashrightarrow \mathcal{Y}  \dashrightarrow \operatorname{Kum}(\mathcal{A}') \,,
\eeq 
where $\mathcal{Y}$ is the multi-parameter generalization of the mirror family $\mathcal{Y}_\mu$. Just as before, this fact implies that the family $\mathcal{Y}$ is a rational cover of a twisted Legendre pencil $\mathcal{Y}'$. We then give an explicit formula for the rational point-count for the family $\mathcal{Y}$; see Theorem~\ref{thm:main2}. By carrying out the rational point-count with respect to either of \emph{two} elliptic fibrations -- resulting from the \emph{two} sandwiches in Equation~\eqref{eqn:sandwich_intro} -- we also derive the transformation of the counting function with respect to $(2,2)$-isogenies of the underlying abelian surface.  Our method also produces an explicit count for the number of rational points on the Jacobian of a genus-two curve over $\mathbb{F}_p$, which has been notoriously difficult to handle based on a traditional group action approach~\cite{MR700577}.
\subsection*{Acknowledgments}
The first author would like to acknowledge the support from an Undergraduate Research and Creative Opportunity (URCO) Grant by the Office of Research and Graduate Studies at Utah State University.  The second author acknowledges support from the Simons Foundation through grant no.~202367.
\bigskip
\section{The mirror-quartic family}
\label{sec:K3_19}
In the case $n=3$ in Equation~\eqref{Fermat}, it was proved in~\cite{MR1877764} that the mirror-quartic $\mathcal{Y}_{\lambda^2}$, considered as a family depending on the parameter $\lambda^2$, is a family of toric Calabi-Yau varieties arising from the polytope $P_0^*$ dual to the simplest reflexive polytope $P_0$ in dimension $3$. The details of the general construction of a family of Calabi-Yau varieties arising from a reflexive polytope can be found in \cite{MR1299003}.  In turn, it was also shown in~\cite{MR1877764} that the quartic surfaces $\mathcal{Z}$ in $\mathbb{P}^3$ are the toric Calabi-Yau varieties arising from the reflexive polytope $P_0$ itself. 
\par The one-dimensional families $\mathcal{X}_{\lambda}$ and $\mathcal{Y}_{\lambda^2}$ can also be described as families of K3 surfaces with canonical lattice polarization. The following is known \cite{MR1420220}:
\begin{lemma}
\label{lem:families19}
The families $\mathcal{X}_{\lambda}$ and $\mathcal{Y}_{\lambda^2}$ are families of lattice polarized K3 surfaces of Picard rank $19$.
In particular, the family $\mathcal{Y}_{\lambda^2}$ is polarized by the rank-19 lattice $M_2 = U \oplus E_8(-1) \oplus E_8(-1) \oplus \langle -4 \rangle$ such that its general member has the N\'eron-Severi lattice $\operatorname{NS}(\mathcal{Y}_{\lambda^2}) \cong M_2$ and the transcendental lattice $\operatorname{T}_{\mathcal{Y}_{\lambda^2}} \cong  U \oplus \langle 4 \rangle$.
\end{lemma}
Here $U$ denotes the hyperbolic rank-two lattice, and $E_8(-1)$ the unique negative definite even unimodular lattice of rank eight. In \cite{MR1420220} Dolgachev established a mathematical framework of mirror symmetry for K3 surfaces in terms of Arnold's strange duality: mirror symmetry for K3 surfaces identifies marked deformations of K3 surfaces $\mathcal{Z}$ with given Picard lattice $N$ with a complexified K\"ahler cone $K(M) = \lbrace x + i y: \, \langle y, y \rangle > 0, \; x, y \in M_\mathbb{R} \rbrace$ for some mirror lattice $M$; for many lattices, one can construct $M$ explicitly by taking a copy of $U$ out of the orthogonal complement $N^\perp$ in the K3 lattice $L \cong U^3 \oplus E_8(-1) \oplus E_8(-1)$.  In the case of the rank-one lattice $N = \langle 2k \rangle$, it turns out that $M \cong  U \oplus E_8(-1) \oplus E_8(-1) \oplus \langle -2k \rangle$ is unique if $k$ has no square divisor. As an application, one considers the smooth quartic surfaces $\mathcal{Z}$ in $\mathbb{P}^3$ with $\operatorname{NS}(\mathcal{Z})=\langle 4 \rangle$: the generic member of the family of K3 surfaces $\mathcal{Y}_{\lambda^2}$ obtained by the Greene-Plesser orbifolding method then obeys Arnold's duality since
\beq
 \operatorname{NS}(\mathcal{Y}_{\lambda^2})^\perp \cong U \oplus \operatorname{NS}(\mathcal{Z}) \;.
\eeq
\par On the other hand, each member of the family of K3 surfaces polarized by the lattice $M_k = U \oplus E_8(-1) \oplus E_8(-1) \oplus \langle -2k \rangle$ (for any $k$) admit what is known as \emph{Shioda-Inose structure} \cites{MR728142, MR0441982}: a K3 surface $\mathcal{Y}$ is said to have a Shioda-Inose structure if it admits a rational map $\phi: \mathcal{Y} \dasharrow \operatorname{Kum}(\mathcal{A})$ of degree two into a Kummer surface $\operatorname{Kum}(\mathcal{A})$, i.e., the K3 surface obtained as the smooth resolution of $\mathcal{A} /\lbrace \pm \mathbb{I} \rbrace$ for some abelian surface $\mathcal{A}$ with inversion automorphism $-\mathbb{I}$, such that the induced map $\phi_*$ is  a Hodge isometry between the transcendental lattices,
\beq
 \phi_*: \quad \operatorname{T}_\mathcal{Y}(2) \to \operatorname{T}_{\operatorname{Kum}(\mathcal{A})} \,.
\eeq 
Morrison proved that $\mathcal{Y}$ admits a Shioda-Inose structure if and only if there exists a Hodge isometry, $\operatorname{T}_\mathcal{Y} \cong \operatorname{T}_\mathcal{A}$, between the transcendental lattices of $\mathcal{Y}$ and an abelian surface $\mathcal{A}$ \cite{MR728142}. Thus, in the situation of the family of K3 surfaces polarized by the lattice $M_k = U \oplus E_8(-1) \oplus E_8(-1) \oplus \langle -2k \rangle$ (for any $k$) each generic member also admits such a Shioda-Inose structure, associated with the abelian surface $\mathcal{A} = \mathcal{E} \times \mathcal{E}'$ where $\mathcal{E}, \mathcal{E}'$ are elliptic curves, and $\mathcal{E}'$ is $k$-isogeneous to $\mathcal{E}$, i.e., $\mathcal{E}' = \mathcal{E}/(\mathbb{Z}/k\mathbb{Z})$ \cite{MR1420220}.
\subsection{Existence of a special elliptic fibration on the mirror-quartic}
\label{ssec:NSfibration}
As a reminder, an elliptic surface is a (relatively) minimal complex surface $\mathcal{Y}$ together with a Jacobian elliptic fibration, that is a holomorphic map $\pi: \mathcal{Y} \to \mathbb{P}^1$ to $\mathbb{P}^1$ such that the general fiber is a smooth curve of genus one together with a distinguished section $\sigma: \mathbb{P}^1 \to \mathcal{Y}$ that marks a smooth point in each fiber. To each Jacobian elliptic fibration $\pi: \mathcal{Y} \to \mathbb{P}^1$ there is an associated Weierstrass model obtained by contracting all components of reducible fibers not meeting $\sigma$. The complete list of possible singular fibers has been given by Kodaira~\cite{MR0165541}.  It encompasses two infinite families $(I_n, I_n^*, n \ge0)$ and six exceptional cases $(II, III, IV, II^*, III^*, IV^*)$. The Weierstrass model of a smooth K3 surface can always be written in the form
\beq
\label{eqn:WEQ}
 Y^2 = 4 X^3 -g_2(u) \, X  - g_3(u)  \,,
\eeq 
where $u$ is a suitable affine coordinate on the base curve $\mathbb{P}^1$, and $g_2$ and $g_3$ are polynomials in $u$ of degree at most $8$ and $12$, respectively. The section is given by the point at infinity in each smooth fiber. We denote the Mordell-Weil group of sections on the Jacobian elliptic surface $\pi: \mathcal{Y} \to \mathbb{P}^1$ by $\operatorname{MW}(\mathcal{Y},\pi,\sigma)$. If a Jacobian elliptic fibration admits in addition a two-torsion section $T \in \operatorname{MW}(\mathcal{Y},\pi,\sigma)$, then we can use a change of coordinates to write Equation~\eqref{eqn:WEQ} in the form
\beq
\label{eqn:WEQ2}
 Y^2 = X^3 + A(u) \, X^2 + B(u) \, X\,,
\eeq 
where $A$ and $B$ are polynomials of degree at most $4$ and $8$, respectively. 
\par In the case $n=3$ in Equation~\eqref{Fermat}, the family $\mathcal{Y}_{\lambda^2}$ is equivalent to a family of such minimal Weierstrass equations. We have the following:
\begin{lemma}\label{lem:fibration}
The family $\mathcal{Y}_{\lambda^2}$ is a family of Jacobian elliptic K3 surfaces given by Equation~\eqref{eqn:WEQ} with the Weierstrass coefficients 
\beq
\label{G2G3}
\begin{split}
 g_2 & = \frac{4}{3 \, \lambda^4} \,{u}^{2} \,  \left( u^4 + 8 \lambda^2 u^3 +(4\lambda^2-1)(4\lambda^2+1)u^2+ 8 \lambda^2 u + 1 \right) \,,\\
 g_3 & = \frac {4}{27  \, \lambda^6} \,{u}^{3} \, \left( u^2 +4 {\lambda}^{2} u +1 \right)  \left( 2 u^4 + 16 \lambda^2 u^3 + (32\lambda^4-5)u^2 + 16 \lambda^2u+2\right) \,.
\end{split}
\eeq
\end{lemma}
\begin{proof}
Following Narumiya and Shiga \cite{MR1877764}, we set
\beq
\label{eqn:WEQcoeffs}
\begin{split}
 x_1 & = -{\frac { \left( 4\,{u}^{2}{\lambda}^{2}+3\,X{\lambda}^{2}+{u}
^{3}+u \right)  \left( 4\,{u}^{2}{\lambda}^{2}+3\,X{\lambda}^{2}+{u}^{
3}-2\,u \right) }{{6\lambda}^{2}u \left( 16\,{u}^{3}{\lambda}^{2}-3\,iY
{\lambda}^{2}+12\,Xu{\lambda}^{2}+4\,{u}^{4}+4\,{u}^{2} \right) }} \,,\\
x_2 & = -\,{\frac {16\,{u}^{3}{\lambda}^{2}-3\,iY{\lambda}^{2}+12\,Xu{
\lambda}^{2}+4\,{u}^{4}+4\,{u}^{2}}{8 u \left( 4\,{u}^{2}{\lambda}^{2}+3
\,X{\lambda}^{2}+{u}^{3}-2\,u \right) }} \,,\\
x_3 & = {\frac {{u}^{2}
 \left( 4\,{u}^{2}{\lambda}^{2}+3\,X{\lambda}^{2}+{u}^{3}-2\,u
 \right) }{{2\lambda}^{2} \left( 16\,{u}^{3}{\lambda}^{2}-3\,iY{\lambda
}^{2}+12\,Xu{\lambda}^{2}+4\,{u}^{4}+4\,{u}^{2} \right) }} \,,
\end{split}
\eeq
in Equation~\eqref{mirror_family} and obtain a Weierstrass equation of the form given by Equation~\eqref{eqn:WEQ} with the coefficients
given by Equations~\eqref{eqn:WEQcoeffs}.
\end{proof}
We also have the following:
\begin{lemma}
For generic parameter $\lambda$ the Jacobian elliptic fibration in Lemma~\ref{lem:fibration} has four singular fibers of Kodaira-type $I_1$, two singular fibers of Kodaira-type $I_4^*$, and the Mordell-Weil group $\mathbb{Z}/2\mathbb{Z} \oplus \langle 1 \rangle$, generated by a two-torsion section and an infinite-order section of height pairing one. 
\end{lemma}
\begin{proof}
The result follows by determining the singular fibers of the Weierstrass model in Equation~\eqref{eqn:WEQ}.  One checks that the Mordell-Weil group contains the two-torsion section $(X,Y)=(-  u \,  (4  \, u- u^2-1)/(3\, \lambda^2),0)$ generating its entire torsion.  We already know from Lemma~\ref{lem:families19} that the family of Jacobian elliptic K3 surfaces has Picard rank $19$. Thus, there must a section of infinite order as well.  The singular fibers and the torsion section of the elliptic fibration contribute to the determinant of the discriminant group a factor of $4^2/2^2=2^2$. Since the determinant of the discriminant group of $M_2$ is $4$, the infinite-order section has to have height pairing one.
\end{proof}
\subsection{Three Kummer sandwich theorems in Picard rank 18}
\label{ssec:Kummer18}
As a reminder, on an algebraic K3 surface $\mathcal{S}$  a {\it Nikulin involution} \cites{MR728142,MR544937} is an involution $\imath: \mathcal{S} \rightarrow \mathcal{S}$ that satisfies $\imath^*\omega_\mathcal{S} = \omega_\mathcal{S}$ for any holomorphic two-form $\omega_\mathcal{S}$ on $\mathcal{S}$. The fixed locus of $\imath$ always consists of eight distinct points. One takes the quotient of $\mathcal{S}$ by the involution $\imath$ and then resolves the eight resulting $A_1$-singularities on the quotient. This procedure, referred in the literature as the {\it Nikulin construction}, leads to a smooth K3 surface $\mathcal{Y}$, birational to $\mathcal{S}/\langle\imath \rangle$ and related to $\mathcal{S}$ via a rational double-cover map $\psi \colon \mathcal{S} \dashrightarrow \mathcal{Y}$.
\par As we will show, each K3 surface in Lemma~\ref{lem:fibration} actually admits a so-called \emph{Kummer Sandwich Theorem}, a term introduced by Shioda \cite{MR2279280}. A K3 surface $\mathcal{Y}$ is a Kummer sandwich if there is a Kummer surface $\mathcal{S}=\operatorname{Kum}(\mathcal{A})$ associated with an abelian surface $\mathcal{A}$ admitting two commuting Nikulin involutions $\imath, \jmath$ such that the quotient surface $\mathcal{S}/\langle \imath \rangle$ is birational to $\mathcal{Y}$, and the quotient surface  $\mathcal{S}/\langle \imath, \jmath \rangle$ is birational to $\mathcal{S}$ itself. Thus, $\mathcal{S}$ dominates and is dominated by $\mathcal{Y}$ by two rational maps of degree two
\beq
 \psi: \; \mathcal{S} \dasharrow \mathcal{Y} \,, \qquad \phi:  \mathcal{Y} \dasharrow \mathcal{S} \,.
 \eeq
\par However, let us first consider a generalization of the situation in Lemma~\ref{lem:fibration}: in \cite{MR1013073}  Oguiso studied the Kummer surface $\operatorname{Kum}(\mathcal{E}_1\times \mathcal{E}_2)$ obtained by the minimal resolution of the quotient surface of the product abelian surface  $\mathcal{A}=\mathcal{E}_1\times \mathcal{E}_2$ by the inversion automorphism, where the elliptic curves $\mathcal{E}_n$ for $n=1,2$ are \emph{not} mutually isogenous. Such Kummer surfaces are algebraic K3 surfaces of Picard rank $18$ and admit Jacobian elliptic fibrations that were classified by Oguiso \cite{MR1013073}: there are eleven inequivalent Jacobian elliptic fibration which we label $\mathcal{J}_1, \dots, \mathcal{J}_{11}$. Kuwata and Shioda furthered Oguiso's work in \cite{MR2409557} and computed elliptic parameters and Weierstrass equations for all eleven fibrations, and analyzed the reducible fibers and Mordell-Weil lattices. 
\par The Weierstrass equations in \cite{MR2409557} define families of minimal Jacobian elliptic fibrations over a two-dimensional moduli space: let $\lambda_n \in \mathbb{P}^1 \backslash \lbrace 0, 1, \infty \rbrace$ with $n=1, 2$ be the modular parameters defining the elliptic curves $\mathcal{E}_n$ using the equations
\beq
\label{eqn:EC}
 \mathcal{E}_n: \quad y_n^2  = x_n \big(x_n-1\big) \big(x_n- \lambda_n\big) \,,
\eeq
with hyperelliptic involutions $\imath_n: (x_n, y_n) \mapsto (x_n,-y_n)$.  The moduli space for the fibrations $\mathcal{J}_1, \dots, \mathcal{J}_{11}$ is then given by unordered pairs of modular parameters for the two elliptic curves with such level-two structure. In this paper, the elliptic fibrations $\mathcal{J}_1$, $\mathcal{J}_4$, $\mathcal{J}_6$, and $\mathcal{J}_7$ will be of particular importance.  Using the Hauptmodul or modular function $\lambda$ of level two for the genus-zero, index-six congruence subgroup $\Gamma(2) \subset \operatorname{PSL}_2(\mathbb{Z})$ we define the moduli space
\beq
\label{eqn:ModuliSpace}
 \mathcal{M} = \Big\{ \{ \lambda_1, \lambda_2 \} \mid \ \lambda_i=\lambda(\tau_i) \,, \tau_i\in \Gamma(2) \backslash \mathbb{H} \; \text{for $i=1,2$}\Big\}\,.
\eeq
When necessary we also consider the fibrations $\mathcal{J}_1, \dots, \mathcal{J}_{11}$ over the covering space of the moduli space $\mathcal{M}$ given by
\beq
\label{eqn:ModuliSpace_cover}
 \widetilde{\mathcal{M}} = \Big\{ (\{ \lambda_1, \lambda_2 \}, l) \mid \  \{\lambda_1,\lambda_2\} \in \mathcal{M} \;, \; l^2 =\lambda_1\lambda_2 \Big\} \,.
\eeq
 \par The simplest fibration on $\mathcal{S}_{\lambda_1, \lambda_2}=\operatorname{Kum}(\mathcal{E}_1\times \mathcal{E}_2)$ is called the \emph{double Kummer pencil}. It is the elliptic fibration with section, denoted by $\mathcal{J}_4$, induced from the projection of the abelian surface $\mathcal{E}_1\times \mathcal{E}_2$ onto its first factor. With the new variable $y_{12}=y_1 y_2$, an affine model of $\mathcal{J}_4$ is then given by the product of two copies of Equation~\eqref{eqn:EC}, re-written as
\beq
\label{eqn:J4}
 y_{12}^2 = x_1 \big(x_1-1\big) \big(x_1 -\lambda_1\big)  x_2 \big(x_2-1\big)\big(x_2-\lambda_2\big) \,,
\eeq
where $x_1$ is considered the affine coordinate of the base curve $\mathbb{P}^1$. The unique holomorphic two-form (up to scaling) on $\mathcal{S}_{\lambda_1, \lambda_2}$ is given by $\omega_\mathcal{S} = dx_1 \wedge dx_2/y_{1,2}$. The quotient variety $(\mathcal{E}_1\times \mathcal{E}_2)/\langle \imath_1 \times \imath_2 \rangle$ is birational to the Kummer surface $\operatorname{Kum}(\mathcal{E}_1\times \mathcal{E}_2)$ since the product involution $-\mathbb{I}=\imath_1 \times \imath_2$ is the inversion involution on the abelian surface $\mathcal{E}_1\times \mathcal{E}_2$. We have the following:
\begin{lemma}\label{lem:J4}
Equation~\eqref{eqn:J4} determines the elliptic fibration with section $\mathcal{J}_4$ on the Kummer surface $\mathcal{S}_{\lambda_1, \lambda_2}=\operatorname{Kum}(\mathcal{E}_1\times \mathcal{E}_2)$ associated with the elliptic curves $\mathcal{E}_1, \mathcal{E}_2$ in Equation~\eqref{eqn:EC}. Generically, the Weierstrass model has four singular fibers of Kodaira-type $I_0^*$ at $x_1=0, 1, \lambda_1, \infty$, and the Mordell-Weil group $(\mathbb{Z}/2\mathbb{Z})^2$.
\end{lemma}
\begin{proof}
The result follows by determining the singular fibers and the Mordell-Weil group of sections of the Weierstrass model in Equation~\eqref{eqn:J4} and comparing these with the results in \cite{MR2409557}.
\end{proof}
\par A birational transformation with $t=y_2 x_1/(y_1x_2)$ given in \cite{MR2409557}*{Sec.~2.1} changes Equation~\eqref{eqn:J4} into the equation
\beq
\label{eqn:J1a}
y^2 = x^3 + \Big((\lambda_1-1)^2 t^4 - 2(\lambda_1+1)(\lambda_2+1) t^2 + ( \lambda_2-1)^2\Big) x^2 + 16 \lambda_1 \lambda_2 t^4 x \,.
\eeq
After a rescaling of the coordinates, we obtain the more symmetric equation
\beq
\label{eqn:J1}
y^2 = x^3 + t^2 \Big((\lambda_1-1)(\lambda_2-1)(t^2+t^{-2}) - 2(\lambda_1+1)(\lambda_2+1)  \Big) x^2 + 16 \lambda_1 \lambda_2 t^4 x \,,
\eeq
with a holomorphic two-form $\omega_\mathcal{S} = dt \wedge dx/y$, where $x, y$ are the affine coordinates of the elliptic fiber, and $t$ is the affine coordinate of the base curve $\mathbb{P}^1$. This establishes the fibration $\mathcal{J}_1$ from \cite{MR2409557} on  $\mathcal{S}_{\lambda_1, \lambda_2}=\operatorname{Kum}(\mathcal{E}_1\times \mathcal{E}_2)$. We make the following:
\begin{remark}
\label{rem:involution1}
An automorphism $\jmath \in \operatorname{Aut}(\mathcal{S}_{\lambda_1, \lambda_2})$ of Equation~\eqref{eqn:J1} is given by 
\beq
\label{eqn:involution1}
\jmath: \;  (t, x, y) \mapsto \left( \frac{1}{t}, \frac{x}{t^4}, - \frac{y}{t^6} \right)\,,
 \eeq
and leaves $\omega_\mathcal{S}$ invariant. Hence, $\jmath$ is a Nikulin involution. 
\end{remark}
\par The following lemma was proved in \cite{MR2409557}*{Sec.~2.1}:
\begin{lemma}
\label{lem:J1}
Equation~\eqref{eqn:J1} determines the elliptic fibration with section $\mathcal{J}_1$ on the Kummer surface $\mathcal{S}_{\lambda_1, \lambda_2}=\operatorname{Kum}(\mathcal{E}_1\times \mathcal{E}_2)$ associated with the elliptic curves $\mathcal{E}_1, \mathcal{E}_2$ in Equation~\eqref{eqn:EC}. Generically, the Weierstrass model has two singular fibers of Kodaira-type $I_8$ at $t=0, \infty$, eight singular fibers of type $I_1$ located over a collection of base points invariant under $t \mapsto -t$ and $t \mapsto 1/t$, and the Mordell-Weil group $\mathbb{Z}^2 \oplus \mathbb{Z}/2\mathbb{Z}$. 
\end{lemma}
\begin{proof}
The proof is analogous to the proof of Lemma~\ref{lem:J1}. It is then easy to show that the Mordell-Weil group of sections is generated by one two-torsion section and two orthogonal infinite-order sections with height pairing $1$.
\end{proof}
Here we mention the Shioda-Tate formula: Using the elliptic fibration in Equation~\eqref{eqn:J1} one checks that the Picard rank of the K3 surface is 18: the sum of the ranks of the reducible fibers plus two for the sub-lattice associated with the elliptic fiber and the section is 16; moreover, the Mordell-Weil group contributes two to the Picard rank from two infinite-order sections. One can then compute the determinant of the discriminant group and obtain $8^2/2^2=16$. Hence, the rank of the transcendental lattice is four and its discriminant group has determinant $16$. This matches precisely the characteristics of the transcendental lattice of a Kummer surface associated with two non-isogeneous elliptic curves which is is isomorphic to $U(2) \oplus U(2)$ where $U(2)$ is the standard rank-two hyperbolic lattice with its quadratic form rescaled by two.
We make the following:
\begin{remark}
\label{rem:involution2} 
There is a second Nikulin involution $\imath \in \operatorname{Aut}(\mathcal{S}_{\lambda_1, \lambda_2})$ given by
\beq
\label{eqn:involution2}
 \imath: \; (t, x, y) \mapsto \left( -t, x, - y \right)\,.
 \eeq
It is obvious that the action of the involution $\imath$ commutes with that of $\jmath$ in Remark~\ref{rem:involution1}. 
\end{remark}
We define a new K3 surface $\mathcal{Y}_{\lambda_1, \lambda_2}$ to be the minimal resolution of $\mathcal{S}_{\lambda_1, \lambda_2}/\langle \imath \rangle$. The quotient map induces the degree-two rational map 
\beq
\label{eqn:psi}
 \psi: \mathcal{S}_{\lambda_1, \lambda_2} \dasharrow \mathcal{Y}_{\lambda_1, \lambda_2}\,, \quad (t, x, y)  \mapsto (u, X, Y)=(t^2, t^2x, t^3y) \,,
 \eeq
such that $\psi^* \omega_\mathcal{Y} = 2 \, \omega_\mathcal{S}$ for the holomorphic two-form $\omega_\mathcal{Y} = du \wedge dX/Y$ on $\mathcal{Y}_{\lambda_1, \lambda_2}$, and a Weierstrass equation for $\mathcal{Y}_{\lambda_1, \lambda_2}$ given by
\beq
\label{eqn:B12}
Y^2 = X^3 + u^2 \Big((\lambda_1-1)(\lambda_2-1)(u+u^{-1}) - 2(\lambda_1+1)(\lambda_2+1)  \Big)  X^2 + 16 \lambda_1 \lambda_2 u^4 X \,.
\eeq
We have the following:
\begin{lemma}\label{lem:B12}
Equation~\eqref{eqn:B12} defines an elliptic fibration with section on the K3 surface $\mathcal{Y}_{\lambda_1, \lambda_2}$. Generically, the Weierstrass model has two singular fibers of Kodaira-type $I_4^*$ at $u=0, \infty$, four singular fibers of type $I_1$ located over a collection of base points invariant under $u \mapsto 1/u$, and the Mordell-Weil group $\mathbb{Z}/2\mathbb{Z}$.
\end{lemma}
\begin{proof}
The proof is analogous to the proof of Lemma~\ref{lem:J1}.
\end{proof}
\par The involution $\jmath$ in Remark~\ref{rem:involution1} induces a Nikulin involution $\jmath'' \in \operatorname{Aut}(\mathcal{Y}_{\lambda_1, \lambda_2})$ of Equation~\eqref{eqn:B12}. Therefore, we obtain another K3 surface $\mathcal{S}''_{\lambda_1, \lambda_2}$ as the minimal resolution of $\mathcal{Y}_{\lambda_1, \lambda_2}/\langle \jmath'' \rangle$ with a degree-two rational map 
\beq
\label{eqn:phi}
 \phi: \mathcal{Y}_{\lambda_1, \lambda_2} \dasharrow \mathcal{S}''_{\lambda_1, \lambda_2}\,, \;\; (u, X, Y)  \mapsto 
 (v, x, y)=\left(u + u^{-1}, \frac{(u^2-1)^2}{u^4} X, \frac{(u^2-1)^3}{u^6} Y\right),
 \eeq
such that $\phi^* \omega_{\mathcal{S}''} =  \omega_\mathcal{Y}$ for the holomorphic two-form $\omega_{\mathcal{S}''} = dv \wedge dx/y$ on $\mathcal{S}''_{\lambda_1, \lambda_2}$, where a Weierstrass equation for $\mathcal{S}''_{\lambda_1, \lambda_2}$ is given by
\beq
\label{eqn:J7}
y^2 = x^3 + \big(v^2 -4\big) \Big((\lambda_1-1)(\lambda_2-1) v - 2(\lambda_1+1)(\lambda_2+1) \Big)  x^2 + 16 \lambda_1 \lambda_2 (v^2-4)^2 x \,.
\eeq
We have the following:
\begin{lemma}
\label{lem:J7}
Equation~\eqref{eqn:J7} determines an elliptic fibration with section on the K3 surface $\mathcal{S}''_{\lambda_1, \lambda_2}$. Generically, the Weierstrass model has a singular fiber of Kodaira-type $I_4^*$ at $v=\infty$, two singular fibers of type $I_0^*$ at $v=\pm 1$, two singular fibers of type $I_1$, and the Mordell-Weil group $\mathbb{Z}/2\mathbb{Z}$.
\end{lemma}
\begin{proof}
The proof is analogous to the proof of Lemma~\ref{lem:J1}.
\end{proof}
We also have the following:
\begin{lemma}
\label{lem:J7moduli}
The K3 surface $\mathcal{S}''_{\lambda_1, \lambda_2}$ is the Kummer surface $\mathcal{S}_{\lambda_1, \lambda_2}=\operatorname{Kum}(\mathcal{E}_1\times \mathcal{E}_2)$.
\end{lemma}
\begin{proof}
Up to rescaling, Equation~\eqref{eqn:J7} is the Weierstrass equation for $\mathcal{J}_7$ on the Kummer surface $\mathcal{S}_{\lambda_1, \lambda_2}=\operatorname{Kum}(\mathcal{E}_1\times \mathcal{E}_2)$ determined in \cite{MR2409557}.
\end{proof}
\par We now prove a Kummer sandwich theorem relating the Jacobian elliptic K3 surface $\mathcal{Y}_{\lambda_1, \lambda_2}$  in Equation~\eqref{eqn:B12} to the  Kummer surface $\mathcal{S}_{\lambda_1, \lambda_2}=\operatorname{Kum}(\mathcal{E}_1\times \mathcal{E}_2)$ associated with the product of the two elliptic curves $\mathcal{E}_1, \mathcal{E}_2$ in Equation~\eqref{eqn:EC}:
\begin{proposition}\label{prop1}
The Kummer surface $\mathcal{S}_{\lambda_1, \lambda_2}$ admits two commuting Nikulin involutions $\imath, \jmath$ (given by Equations~\eqref{eqn:involution1} and~\eqref{eqn:involution2}) such that the quotient surface $\mathcal{S}_{\lambda_1, \lambda_2}/\langle \imath \rangle$ is birational to $\mathcal{Y}_{\lambda_1, \lambda_2}$, and the quotient surface  $\mathcal{S}_{\lambda_1, \lambda_2}/\langle \imath, \jmath \rangle$ is birational to $\mathcal{S}_{\lambda_1, \lambda_2}$ itself. Thus, $\mathcal{S}_{\lambda_1, \lambda_2}$ dominates and is dominated by $\mathcal{Y}_{\lambda_1, \lambda_2}$ via the two rational maps of degree two
\beq
 \psi: \; \mathcal{S}_{\lambda_1, \lambda_2} \dasharrow \mathcal{Y}_{\lambda_1, \lambda_2} \,, \qquad \phi:  \mathcal{Y}_{\lambda_1, \lambda_2} \dasharrow \mathcal{S}_{\lambda_1, \lambda_2} \,,
 \eeq
given by Equations~\eqref{eqn:psi} and~\eqref{eqn:phi}, such that $2 \, \omega_\mathcal{S} = \psi^* \omega_\mathcal{Y} $ and $\omega_\mathcal{Y} = \phi^* \omega_{\mathcal{S}} $ for the holomorphic two-forms $\omega_\mathcal{Y}$ and $\omega_\mathcal{S}$ on $\mathcal{Y}_{\lambda_1, \lambda_2}$ and $\mathcal{S}_{\lambda_1, \lambda_2}$, respectively.
\end{proposition}
\begin{proof}
The proof follows from the sequence of arguments provided by Lemma~\ref{lem:J1}, Remark~\ref{rem:involution2}, Lemma~\ref{lem:B12}, Remark~\ref{rem:involution1}, and Lemmas~\ref{lem:J7} and~\ref{lem:J7moduli}.
\end{proof}
\par {\it Van Geemen-Sarti involutions} \cites{MR2274533,MR2824841} provide a particularly interesting case of Nikulin constructions. In this situation, the K3 surface $\mathcal{Y}$ is endowed with a Jacobian elliptic  fibration $\pi \colon \mathcal{Y} \rightarrow \mathbb{P}^1 $ which, in addition to the trivial section $\sigma$, carries an additional section $T$ that makes an element of order two in the Mordell-Weil group $\operatorname{MW}(\mathcal{Y}, \pi,\sigma)$. Fiberwise translations by the order-two section $T$ are then known to define an involution on $k \in \operatorname{Aut}(\mathcal{Y})$ on $\mathcal{Y}$ -- the Van Geemen-Sarti involution. These involutions are special Nikulin involutions; see \cite{MR3995925}.
\par The Jacobian elliptic K3 surface $\mathcal{Y}_{\lambda_1, \lambda_2}$ in Proposition~\ref{prop1} admits the two-torsion section $T: (X,Y)=(0,0)$. It is straight forward to show (the details are contained in the proof of Proposition~\ref{prop2}) that the minimal resolution of $\mathcal{Y}_{\lambda_1, \lambda_2}/\langle k \rangle$ is a Jacobian elliptic K3 surface $\mathcal{S}'_{\lambda_1, \lambda_2}$ with the Weierstrass equation
\beq
\label{eqn:J6dual}
\begin{split}
 y^2 = & \, x^3 - \frac{1}{2} \, u^2 \Big((\lambda_1-1)(\lambda_2-1)(u+u^{-1}) - 2(\lambda_1+1)(\lambda_2+1)  \Big)  x^2 \\
 & + \frac{1}{16} \, u^4 \Big( \big((\lambda_1-1)(\lambda_2-1)(u+u^{-1}) - 2(\lambda_1+1)(\lambda_2+1) \big)^2 -64  \lambda_1 \lambda_2 \Big) x \,.
\end{split} 
\eeq
We have the following:
\begin{lemma}
\label{lem:J6dual}
Equation~\eqref{eqn:J6dual} determines an elliptic fibration with section on the K3 surface $\mathcal{S}'_{\lambda_1, \lambda_2}$. Generically, the Weierstrass model has two singular fibers of Kodaira-type $I_2^*$ at $u=0, \infty$, four singular fibers of type $I_2$ located over a collection of base points invariant under $u \mapsto 1/u$, and the Mordell-Weil group  $(\mathbb{Z}/2\mathbb{Z})^2$.
\end{lemma}
\begin{proof}
The proof is analogous to the proof of Lemma~\ref{lem:J1}.
\end{proof}
\begin{remark}
\label{rem:duality}
The elliptic curve $\mathcal{E}_n$ for $n=1, 2$ in Equation~\eqref{eqn:EC} has the two-torsion point $p: (x,y)=(0,0)$. A two-isogenous elliptic curve $\mathcal{E}'_n$ is given by
\beq
\label{EllC1_dual}
 \mathcal{E}'_n=\mathcal{E}_n/\langle p \rangle: \qquad Y_n^2 = X_n^3  + \frac{1+\lambda_n}{2} \, X_n^2 +\frac{(1-\lambda_n)^2}{4} \, X_n\,.
\eeq
The two-isogenous elliptic curve $\mathcal{E}'_n$ has the two-torsion point $p': (X,Y)=(0,0)$ such that $\mathcal{E}''_n=\mathcal{E}'_n/\langle p' \rangle$ is isomorphic to $\mathcal{E}_n$. Moreover, we denote by $\Psi: \mathcal{A}=\mathcal{E}_1 \times \mathcal{E}_2 \to \mathcal{A}'=\mathcal{E}'_1 \times \mathcal{E}'_2$ the product-isogeny. As map between abelian surfaces $\Psi$ is a $(2,2)$-isogeny. Similarly, there is a dual $(2,2)$-isogeny $\Psi': \mathcal{A}' \to \mathcal{A}''\cong \mathcal{A}$. It was shown in \cite{MR1406090} that $\Psi$ and $\Psi'$ descend to rational maps between the associated Kummer surfaces.
\end{remark}
\par We then have the following:
\begin{lemma}
\label{lem:J6moduli}
The K3 surface $\mathcal{S}'_{\lambda_1, \lambda_2}$ is the Kummer surface $\operatorname{Kum}(\mathcal{E}'_1\times \mathcal{E}'_2)$ associated with the two-isogenous elliptic curves $\mathcal{E}'_1$ and $\mathcal{E}'_2$ in Equation~\eqref{EllC1_dual}. 
\end{lemma}
\begin{proof}
The six values of the cross-ratio of the four ramification points are given by
\beq
 \left\lbrace \lambda'_n, \; 1-\lambda'_n, \; \frac{1}{\lambda'_n}, \; 1-\frac{1}{\lambda'_n}, \; \frac{1}{1-\lambda'_n}, \; \dfrac{1}{1-\frac{1}{\lambda'_n}} \right \rbrace \,,
\eeq
with
\beq
 \lambda'_n= \left(\frac{1-\sqrt{\lambda_n}}{1+\sqrt{\lambda_n}} \right)^2 \,.
\eeq
Thus, if we introduce the Jacobi moduli $k_n$ and $k'_n$ with $\lambda_n=k_n^2$ and $\lambda_n'=(k_n')^2$, respectively, satisfying the symmetric relations
\beq
\label{def_M}
 k_n'=  \frac{1-k_n}{1+k_n}  , \qquad  k_n=  \frac{1-k_n'}{1+k_n'}   \,,
\eeq
it follows that $\mathcal{E}'_n$ is isomorphic to the elliptic curve with $\lambda$-parameter $\lambda'_n$.
\par The Weierstrass equation for $\mathcal{J}_6$ computed in \cite{MR2409557} on the Kummer surface $\mathcal{S}'_{\lambda_1, \lambda_2}=\operatorname{Kum}(\mathcal{E}'_1\times \mathcal{E}'_2)$ with $\lambda'_1= (k'_1)^2$ and $\lambda'_2= (k'_2)^2$ is given by
\beq
 Y^2 = X \Big(X - U(U-1)(\lambda'_2U-\lambda'_1)\Big) \Big(X - U(U-\lambda'_1)(\lambda'_2U-1)\Big) \,.
\eeq 
Setting $U=\frac{k'_1 u}{k'_2}$, $X=4x/(\lambda'_2)^2$, $Y=8y/(\lambda'_2)^3$ we obtain Equation~\eqref{eqn:J6dual}.
\end{proof}
\par We now prove a second Kummer sandwich theorem for the Jacobian elliptic K3 surface $\mathcal{Y}_{\lambda_1, \lambda_2}$ in Equation~\eqref{eqn:B12} and the Kummer surface $\mathcal{S}'_{\lambda_1, \lambda_2}=\operatorname{Kum}(\mathcal{E}'_1\times \mathcal{E}'_2)$ associated with the product of the two-isogenous elliptic curves $\mathcal{E}'_n$ in Equation~\eqref{EllC1_dual}:
\begin{proposition}\label{prop2}
The K3 surface $\mathcal{Y}_{\lambda_1, \lambda_2}$ and the Kummer surface $\mathcal{S}'_{\lambda_1, \lambda_2}$ admit dual Van Geemen-Sarti involution $k$ and $k'$ associated with fiberwise translations by the order-two section $T: (X,Y)=(0,0)$ and $T': (x,y)=(0,0)$, respectively, and a pair of dual geometric two-isogenies
\begin{equation}
\label{isog_middle}
 \xymatrix 
{ \mathcal{Y} \ar @(dl,ul) _{k}
\ar @/_0.5pc/ @{-->} _{\varphi} [rr] &
& \mathcal{S}' \ar @(dr,ur) ^{k'}
\ar @/_0.5pc/ @{-->} _{\varphi'} [ll] \\
} 
\end{equation}
such that $\omega_\mathcal{Y} = \varphi^* \omega_{\mathcal{S}'}$ and $2\, \omega_{\mathcal{S}'} =  (\varphi')^* \omega_{\mathcal{Y}}$ for the holomorphic two-forms $\omega_{\mathcal{S}'} = du \wedge dx/y$ and $\omega_\mathcal{Y}=du \wedge dX/Y$ on  $\mathcal{S}'_{\lambda_1, \lambda_2}$ and $\mathcal{Y}_{\lambda_1, \lambda_2}$, respectively.
\end{proposition}
\begin{proof}
We denote the Weierstrass equation for $\mathcal{Y}_{\lambda_1, \lambda_2}$ by $Y^2= X^3 + p_2 X^2+p_1 X$. The translation by the two-torsion section $T: (X,Y)=(0,0)$ is a Nikulin involution $k$, which for $X\not = 0$ is given by addition with respect to the group law in the elliptic fiber, i.e.,
\beq
\label{VGS_involution_middle}
 k: \; (X,Y)  \mapsto  (X,Y) \overset{.}{+}  (0,0) = \left( \frac{p_1}{X}, - \frac{p_1 Y}{X^2} \right) \,.
\eeq
By resolving the eight nodes of  $\mathcal{Y}_{\lambda_1, \lambda_2}/\langle k \rangle$ we obtain a K3 surface equipped with a Jacobian elliptic fibration given by the Weierstrass model 
\beq
\label{kummer_middle_ell_dual_W}
 y^2 = x^3 - \frac{1}{2} \, p_2 \, x^2  + \frac{1}{4} \left( \frac{p_2^2}{4} - p_1\right)  x \,,
\eeq
which is Equation~\eqref{eqn:J6dual} defining $\mathcal{S}'_{\lambda_1, \lambda_2}$. One constructs the dual Van Geemen-Sarti involution $k'$ analogously. The explicit formulas for the isogeny and the dual isogeny are well known and given by
\beq
\label{Eq:isogeny}
\varphi: \quad \mathcal{Y}_{\lambda_1, \lambda_2} \dashrightarrow \mathcal{S}'_{\lambda_1, \lambda_2}\,, \qquad  (X,Y) \mapsto \left( \frac{Y^2}{4 X^2},  \frac{Y \, \big(X^2-p_1 \big)}{8 X^2}\right)\,,
\eeq
and
\beq
\label{eqn:dual_isog}
\varphi':  \quad \mathcal{S}'_{\lambda_1, \lambda_2} \dashrightarrow \mathcal{Y}_{\lambda_1, \lambda_2}\,, \qquad  (x,y) \mapsto \left( \frac{y^2}{x^2}, \frac{y \, \big(16 \, x^2 - p_2^2 +4 \, p_1)  \big)}{16 \, x^2}\right)\,.
\eeq
The rest of the statement follows from Lemma~\ref{lem:J6moduli} and a direct computation.
\end{proof}
\par The involution $\jmath''$ induces a Nikulin involution $\jmath' \in \operatorname{Aut}(\mathcal{S}'_{\lambda_1, \lambda_2})$ of Equation~\eqref{eqn:J6dual}.  We obtain a K3 surface $\mathcal{Y}'_{\lambda_1, \lambda_2}$ as the minimal resolution of $\mathcal{S}'_{\lambda_1, \lambda_2}/\langle \jmath' \rangle$. The quotient map induces the degree-two rational map 
\beq
\label{eqn:phi_prime}
 \phi': \mathcal{S}'_{\lambda_1, \lambda_2} \dasharrow \mathcal{Y}'_{\lambda_1, \lambda_2}\,, \; (u, x, y)  \mapsto (v, X, Y)=\left(u + u^{-1}, \frac{(u^2-1)^2}{u^4} x, \frac{(u^2-1)^3}{u^6} y\right),
 \eeq
such that $(\phi')^* \omega_{\mathcal{Y}'} =  \omega_{\mathcal{S}'}$ for the holomorphic two-form $\omega_{\mathcal{Y}'} = dv \wedge dX/Y$ on $\mathcal{Y}'_{\lambda_1, \lambda_2}$, where a Weierstrass equation for $\mathcal{Y}'_{\lambda_1, \lambda_2}$ is
\beq
\label{eqn:Zfib}
\begin{split}
Y^2 = & \, X^3 - \frac{1}{2} \big(v^2 -4\big) \Big((\lambda_1-1)(\lambda_2-1) v - 2 \, (\lambda_1+1)(\lambda_2+1) \Big)  X^2 \\
& + \frac{1}{16} \,  (v^2-4)^2 \Big( \big((\lambda_1-1)(\lambda_2-1) v - (\lambda_1+1)(\lambda_2+1) \big)^2- 64 \lambda_1 \lambda_2 \Big)  X \,.
\end{split}
\eeq
We have the following:
\begin{lemma}
\label{lem:Zfib}
Equation~\eqref{eqn:Zfib} determines an elliptic fibration with section on the K3 surface $\mathcal{Y}'_{\lambda_1, \lambda_2}$. Generically, the Weierstrass model has a singular fiber of Kodaira-type $I_2^*$ at $v=\infty$, two singular fibers of type $I_0^*$ at $v=\pm 1$, two singular fibers of type $I_2$, and the Mordell-Weil group $(\mathbb{Z}/2\mathbb{Z})^2$.
\end{lemma}
\begin{proof}
The proof is analogous to the proof of Lemma~\ref{lem:J7}.
\end{proof}
\par We also have a Kummer sandwich theorem for the Jacobian elliptic K3 surface $\mathcal{Y}'_{\lambda_1, \lambda_2}$ in Equation~\eqref{eqn:Zfib} and the Kummer surface $\mathcal{S}_{\lambda_1, \lambda_2}=\operatorname{Kum}(\mathcal{E}_1\times \mathcal{E}_2)$ associated with the product of the two elliptic curves $\mathcal{E}_1, \mathcal{E}_2$ in Equation~\eqref{eqn:EC} :
\begin{corollary}
\label{prop3}
The K3 surface $\mathcal{Y}'_{\lambda_1, \lambda_2}$ and the Kummer surface $\mathcal{S}_{\lambda_1, \lambda_2}$ admit dual Van Geemen-Sarti involution $k$ and $k'$ associated with fiberwise translations by the order-two section $T': (X,Y)=(0,0)$ and $T: (x,y)=(0,0)$, respectively, and a pair of dual geometric two-isogenies
\begin{equation}
\label{isog_right}
 \xymatrix 
{ \mathcal{Y}' \ar @(dl,ul) _{k'}
\ar @/_0.5pc/ @{-->} _{\chi'} [rr] &
& \mathcal{S} \ar @(dr,ur) ^{k}
\ar @/_0.5pc/ @{-->} _{\chi} [ll] \\
} 
\end{equation}
such that $\omega_\mathcal{S} = \chi^* \omega_{\mathcal{Y}'}$ and $2 \, \omega_{\mathcal{Y}'} =  (\chi')^* \omega_{\mathcal{S}}$ for the holomorphic two-forms $\omega_{\mathcal{Y}'} = dv \wedge dx/y$ and $\omega_\mathcal{S}=du \wedge dX/Y$ on  $\mathcal{S}_{\lambda_1, \lambda_2}$ and $\mathcal{Y}'_{\lambda_1, \lambda_2}$, respectively.
\end{corollary}
\begin{proof}
The proof is analogous to the proof of Theorem~\ref{prop2}.
\end{proof}
We also have the following:
\begin{lemma}
Over $\widetilde{\mathcal{M}}$ with $l^2=\lambda_1 \lambda_2$, the Jacobian elliptic K3 surface $\mathcal{Y}'_{\lambda_1, \lambda_2}$  in Equation~\eqref{eqn:Zfib} is isomorphic to the twisted Legendre pencil
\beq
\label{eqn:Legendre18}
  y^2 = 16\, l \, x \, (x-1) \, (x- w)  \left( w - \frac{(l+1)^2}{4 l} \right)   \left( w - \frac{(l + \lambda_1)^2}{4 l \lambda_1}\right)  \,,
\eeq
equipped with the holomorphic two-form $dw \wedge dx/y$.
\end{lemma}
\begin{proof}
Starting with Equation~\eqref{eqn:Zfib} one uses a scaling and shift to obtain $w$ and $x$ from $v$ and $X$, respectively.
\end{proof}
\par We summarize the results of Propositions~\ref{prop1} and~\ref{prop2}, Corollary~\ref{prop3} and Remark~\ref{rem:duality} in the following diagram:
\begin{figure}[ht]
\centerline{
\xymatrix{
*+[F--]{\mathcal{S}_{\lambda_1, \lambda_2} = \operatorname{Kum}(\mathcal{E}_1 \times \mathcal{E}_2)}  \ar@{-->}[r]^{\psi}   \ar@{->}[rd]^{\Psi}
	& *+[F-,]{\mathcal{Y}_{\lambda_1, \lambda_2}} \ar@{-->}[r]^{\phi}  \ar@{-->}[d]^{\varphi} 
	& *+[F--]{\mathcal{S}_{\lambda_1, \lambda_2} = \operatorname{Kum}(\mathcal{E}_1 \times \mathcal{E}_2)}  \ar@{-->}[d]^{\chi} \\
	& *+[F--]{\mathcal{S}'_{\lambda_1, \lambda_2} = \operatorname{Kum}(\mathcal{E}'_1 \times \mathcal{E}'_2)} \ar@{-->}[r]^{\phi'}  \ar@{-->}[u]<1ex>^{\varphi'}  \ar@{->}[ru]^{\Psi'}
	& *+[F-,]{\mathcal{Y}'_{\lambda_1, \lambda_2}} \ar@{-->}[u]<1ex>^{\chi'} 
}}
\caption{\label{fig1}}
\end{figure}
\subsection{The Greene-Plesser construction and Kummer sandwiches}
We will now use Propositions~\ref{prop1} and~\ref{prop2} and Corollary~\ref{prop3} in the special situation that the two elliptic curves are two-isogeneous, and the Picard rank of the K3 surfaces increases from 18 to 19. We have the following:
\begin{lemma}
\label{lem:2isog}
The Jacobian elliptic K3 surface $\mathcal{Y}_{\lambda_1, \lambda_2}$ in Equation~\eqref{eqn:B12} coincides with the mirror-quartic family $\mathcal{Y}_{\lambda^2}$ in Lemma~\ref{lem:fibration} if and only if $\mathcal{E}_1$ and $\mathcal{E}_2$ in Equation~\eqref{eqn:EC} are two-isogeneous elliptic curves, i.e., $\mathcal{E}_2 \cong \mathcal{E}'_1$. In particular, $\mathcal{E}_1$ and $\mathcal{E}_2$ in Equation~\eqref{eqn:EC} are two-isogeneous elliptic curves if
\beq
\label{eqn:2isog}
0 =  \lambda_1^2 \lambda_2^2- 2 \lambda_1\lambda_2(\lambda_1+\lambda_2) + \lambda_1^2+\lambda_2^2 -12 \lambda_1\lambda_2 -2(\lambda_1+\lambda_2)+1  \,,
\eeq
and $\lambda^2=-(\lambda_1+1)(\lambda_2+1)/[2(\lambda_1-1)(\lambda_2-1)]$.
\end{lemma}
\begin{proof}
The proof follows by computing and matching the discriminants for Jacobian elliptic fibration in Lemma~\ref{lem:fibration} and for Equation~\eqref{eqn:B12}.
\end{proof}
\begin{remark}
If one considers all holomorphic transformations of the elliptic curves $\mathcal{E}_n$ in Equation~\eqref{eqn:EC}, the condition $\mathcal{E}_2 \cong \mathcal{E}'_1$ in Equation~\eqref{EllC1_dual} can be expressed in terms of the $j$-invariants $j_1$ and $j_2$ of the elliptic curves $\mathcal{E}_1$ and $\mathcal{E}_2$, given by
\beq
 j_k = \frac{256 (\lambda_k^2-\lambda_k +1)^3}{\lambda_k^2 (\lambda_k-1)^2} \,.
\eeq
The modular completion of Equation~\eqref{eqn:2isog} then is the classical modular curve $X_0(2)$ given by
\beq
\begin{split}
 0 = & - j_1^2 j_2^2 + j_1^3 + j_2^3 + 1488 j_1 j_2 (j_1 + j_2) - 2^4 3^4 5^3 (j_1^2+j_2^2)  \\
 &+ 3^4 5^3 4027 j_1 j_2+ 2^7 3^7 5^5 (j_1 + j_2) - 2^{12} 3^9 5^9 \,,
\end{split} 
\eeq
an irreducible plane algebraic curve of genus zero. $X_0(2)$ has a rational parametrization given by $j_1 = (h +256)^3/h^2$ and $j_2=(h+16)^3/h$ for $h \in \mathbb{C}^\times$; see \cite{MR1512996}.
\end{remark}
In the case $n=3$ in Equation~\eqref{Fermat}, the Greene-Plesser orbifolding method constructs from the quartic Dwork pencil $\mathcal{X}_\lambda$ the mirror-quartic family $\mathcal{Y}_{\lambda^2}$. The latter is equipped with a Jacobian elliptic fibration established in Lemma~\ref{lem:fibration}. We then have the following:
\begin{theorem}
\label{prop4}
The Greene-Plesser construction for K3 surfaces $\mathcal{X}_\lambda \dashrightarrow \mathcal{Y}_{\lambda^2}$ factors through a Kummer sandwich. In particular, there are rational maps
\beq
 \mathcal{X}_\lambda \dashrightarrow \operatorname{Kum}(\mathcal{E} \times \mathcal{E}') \overset{\psi'}{\dashrightarrow} \mathcal{Y}_{\lambda^2} 
 \overset{\phi}{\dashrightarrow} \operatorname{Kum}(\mathcal{E} \times \mathcal{E}')  \overset{\chi}{\dashrightarrow} 
 \mathcal{Y}'_{\lambda^2} \,,
\eeq 
where $\mathcal{E}=\mathcal{E}_1$ and $\mathcal{E}'=\mathcal{E}'_1$ are the two-isogeneous elliptic curves given by Equation~\eqref{eqn:EC} and Equation~\eqref{EllC1_dual}, respectively, and the Jacobian elliptic K3 surface $\mathcal{Y}'_{\lambda^2}$ is the twisted Legendre pencil given by
\beq
\label{eqn:Legendre19}
  y^2 = 16 \, l \, x \, (x-1) \, (x- w) \, \left( w - \lambda^2 \right)   \left( w - 1 - \lambda^2\right)  \,,
\eeq
with $l=1+2\lambda\sqrt{\lambda^2-1}-2\lambda^2$. Here, we have equipped $\mathcal{Y}_{\lambda^2}$ and $\mathcal{Y}'_{\lambda^2}$  with the holomorphic two-forms $\omega_\mathcal{Y} = du \wedge dX/Y$ and $\omega_{\mathcal{Y}'}=dw \wedge dx/y$, respectively, such that $\omega_\mathcal{Y} = \phi^*\chi^* \omega_{\mathcal{Y}'}$. \end{theorem}
\begin{proof}
The proof follows by specializing the results of Theorems~\ref{prop1} and~\ref{prop2} and Corollary~\ref{prop3} to the situation where the elliptic curve $\mathcal{E}_2$ is the two-isogeneous elliptic curve $\mathcal{E}'_1$.
Lemma~\ref{lem:2isog} then allows us to express the relevant coefficients in terms of $\lambda$. In Figure~\ref{fig1} for $\mathcal{E}_2=\mathcal{E}'_1$ we have $\mathcal{S}'_{\lambda_1, \lambda_2} =\mathcal{S}_{\lambda_1, \lambda_2}$. Thus, any rational map that dominates $\mathcal{Y}_{\lambda_1, \lambda_2}$ also dominates $\mathcal{S}_{\lambda_1, \lambda_2}$. Finally, Equation~\eqref{eqn:Legendre18} restricted to the situation in Lemma~\ref{lem:2isog} yields the the twisted Legendre pencil in Equation~\eqref{eqn:Legendre19}. 
\end{proof}
\begin{remark}
Theorem~\ref{prop4} implies that the holomorphic solution of the Picard-Fuchs equation of $(\mathcal{Y}_{\lambda^2}, \omega_\mathcal{Y})$ coincide with the one of $(\mathcal{Y}'_{\lambda^2}, \omega_{\mathcal{Y}'})$. Using the result in \cite{MR3767270}*{Thm~2.5} with $\mu=1/2$, $a=\lambda^2$ and $b=1+\lambda^2$, we immediately conclude that the holomorphic solution of the Picard-Fuchs equation for the latter is
\beq
\lambda \; \frac{(2 \, \pi \, i)^2}{\lambda} \, \left( 
\hpg21{  \frac{1}{4}, \, \frac{3}{4} }{1}{A} \right)^2 
\eeq
with $A=\frac{1}{2} \left( 1 -\sqrt{1- 1/\lambda^4}\right)$. The quadratic relation for hypergeometric functions
\beq
\hpg21{  p, q }{ p + q + \frac{1}{2} }{ \frac{1}{\lambda^4} }
 = \hpg21{  2p, 2q}{ p+ q + \frac{1}{2}} {\frac{1}{2} \left( 1 - \sqrt{1-\frac{1}{\lambda^4}}\right) } \,,
\eeq
then implies that the holomorphic period equals
\beq
\label{res1}
  (2 \pi i)^2  \, \left(  \hpg21{\frac{1}{8}, \, \frac{3}{8}}{1}{  \frac{1}{\lambda^4}} \right)^2 
 =  (2 \pi  i)^2   \, \left(  \hpg21{\frac{1}{8}, \, \frac{3}{8}}{1}{\mu} \right)^2 \;,
\eeq
which agrees with the holomorphic period computed by Narumiya and Shiga in~\cite{MR1877764}.
\end{remark}
\section{General quartic Kummer surfaces}
\label{sec:K3_17}
We now turn to the Kummer variety associated with a principally polarized abelian variety; see \cite{MR2062673}. Let $\mathcal{L}$ be the ample line symmetric bundle on an abelian surface $\mathcal{A}$ defining its  principal polarization and consider the rational map $\mathcal{A} \to \mathbb{P}^3$ associated with the line bundle $\mathcal{L}^2$. Its image is easily seen to be a quartic surface in $\mathbb{P}^3=\mathbb{P}(X_0,\dots,X_3)$ given by
\begin{small}
\begin{gather}
\label{Eq:QuarticSurfaces12}
    0 = \xi_0 \, (X_0^4+X_1^4+X_2^4+X_3^4)  + \xi_4 \, X_0 X_1 X_2 X_3 \qquad \\
\nonumber
    + \, \xi_1 \, \big(X_0^2 X_1^2+X_2^2 X_3^2\big)  +\xi_2 \, \big(X_0^2 X_2^2+X_1^2 X_3^2\big)  +\xi_3 \, \big(X_0^2 X_3^2+X_1^2 X_2^2\big) \,,
\end{gather}
\end{small}%
with $[\xi_0:\xi_1:\xi_2:\xi_3:\xi_4] \in \mathbb{P}^4$. Any general member of the family~\eqref{Eq:QuarticSurfaces12} is smooth. As soon as the surface is singular at a general point, it must have sixteen singular nodal points because of its symmetry. The discriminant turns is a homogeneous polynomial of degree eighteen in the parameters $[\xi_0:\xi_1:\xi_2:\xi_3:\xi_4] \in \mathbb{P}^4$ and was determined in \cite{MR2062673}*{Sec.~7.7 (3)}.  The Kummer surfaces form an open set among the hypersurfaces in Equation~\eqref{Eq:QuarticSurfaces12} with parameters $[\xi_0:\xi_1:\xi_2:\xi_3:\xi_4] \in \mathbb{P}^4$, such that the irreducible factor of degree three in the discriminant vanishes, i.e.,
\beq
 \xi_0 \, \big( 16 \xi_0^2 - 4 \xi_1^2-4 \xi_2^2 - 4 \xi_3^3+ \xi_4^2\big) + 4 \, \xi_1 \xi_2 \xi_3 =0 \,.
\eeq
Setting $\xi_0=1$ and using the affine moduli $\xi_1=-A$, $\xi_2=-B$, $\xi_3=-C$, $\xi_4=2 D$, we obtain the normal form of a nodal quartic surface. 
\begin{definition}
A general Kummer quartic is the surface in $\mathbb{P}^3=\mathbb{P}(X_0,\dots,X_3)$ is
\begin{small}
\begin{gather}
\label{Goepel-Quartic}
  0 = X_0^4+X_1^4+X_2^4+X_3^4  + 2  D  X_0 X_1 X_2 X_3\\
\nonumber
    - A  \big(X_0^2 X_1^2+X_2^2 X_3^2\big)  - B \big(X_0^2 X_2^2+X_1^2 X_3^2\big) - C  \big(X_0^2 X_3^2+X_1^2 X_2^2\big)   \;, 
\end{gather}
\end{small}%
where $A, B, C, D \in \mathbb{C}$ such that
\beq
\label{paramGH}
 D^2 = A^2 + B^2 + C^2 + ABC - 4\;.
\eeq
\end{definition}
\begin{remark}
\label{rem:symmetries}
The symmetries of Equation~\eqref{Goepel-Quartic} are generated by the discrete group $(\mathbb{Z}/2\, \mathbb{Z})^2$ of transformations changing signs of two coordinates,  i.e.,
\beq
\label{eqn:reflections}
\begin{split}
 [X_0:X_1:X_2:X_3] & \to [-X_0:-X_1:X_2:X_3] \,,\\ 
 [X_0:X_1:X_2:X_3] & \to [-X_0:X_1: -X_2:X_3] \,,
\end{split}
\eeq
and the permutations, generated by $ [X_0:X_1:X_2:X_3] \to  [X_1:X_0:X_3:X_2]$ and $ [X_0:X_1:X_2:X_3] \to  [X_2:X_3:X_0:X_1]$. 
\end{remark}
\subsection{Moduli of a general Kummer surface}
If $\mathcal{A}=\operatorname{Jac}(\mathcal{C})$ is the Jacobian of a smooth curve $\mathcal{C}$ of genus two, then the hermitian form associated to the divisor class $[\mathcal{C}]$ is a polarization of type $(1, 1)$, also called a principal polarization. Conversely, over the complex numbers a principally polarized abelian surface is either the Jacobi variety of a smooth curve of genus two with the theta-divisor or the product of two complex elliptic curves with the product polarization~\cite{MR2062673}*{Sec.~4}. Therefore, a general Kummer surface is associated with a principally polarized abelian surface $\mathcal{A}=\operatorname{Jac}(\mathcal{C})$ for a smooth curve $\mathcal{C}$ of genus two.
\par We start with a smooth genus-two curve $\mathcal{C}$ in Rosenhain normal form 
\beq
\label{Eq:Rosenhain}
 \mathcal{C}: \quad y^2 = x\,  \big(x-1) \, \big(x- \lambda_1\big) \,  \big(x- \lambda_2 \big) \,  \big(x- \lambda_3\big) \,.
\eeq 
We denote the hyperelliptic involution on $\mathcal{C}$ by $\imath_\mathcal{C}$.  The ordered tuple $(\lambda_1, \lambda_2, \lambda_3)$ -- where the $\lambda_i$ are pairwise distinct and different from $(\lambda_4,\lambda_5,\lambda_6)=(0, 1, \infty)$ -- determines a point in the moduli space $\mathcal{M}_2$ of genus-two curves with level-two structure. We also introduce the moduli $\Lambda_1 = ( \lambda_1 + \lambda_2 \lambda_3)/l$,  $\Lambda_2 = (\lambda_2 + \lambda_1 \lambda_3)/l$,  and  $\Lambda_3 = (\lambda_3 + \lambda_1 \lambda_2)/l$ on the double cover $\widetilde{\mathcal{M}}_2$ of the moduli space given by $l^2 = \lambda_1 \lambda_2 \lambda_3$.
\par The symmetric product of $\mathcal{C}$ is given by $\mathcal{C}^{(2)} = (\mathcal{C}\times\mathcal{C})/\langle \sigma_{\mathcal{C}^{(2)} } \rangle$ where $\sigma_{\mathcal{C}^{(2)}}$ interchanges the copies of $\mathcal{C}$. The variety $\mathcal{C}^{(2)}/\langle \imath_\mathcal{C} \times  \imath_\mathcal{C} \rangle$ is given in terms of the variables $U=x^{(1)}x^{(2)}$, and $X=x^{(1)}+x^{(2)}$, and $Y=y^{(1)}y^{(2)}$ by the affine equation
\beq
\label{kummer_middle}
  Y^2 = U \big(  U  - X +  1 \big)  \prod_{i=1}^3 \big( \lambda_i^2 \, U  -  \lambda_i \, X +  1 \big) \;,
\eeq
equipped with the canonical holomorphic two-form
\beq
\label{eqn:2form}
\operatorname{pr}^*\left( l \, dU\wedge \frac{dX}{Y} \right) = l\, \left( \frac{dx^{(1)}}{y^{(1)}} \boxtimes  \frac{x^{(2)} dx^{(2)}}{y^{(2)}} -  \frac{x^{(1)} dx^{(1)}}{y^{(1)}} \boxtimes \frac{dx^{(2)}}{y^{(2)}} \right) \,,
\eeq 
where $\operatorname{pr}:\mathcal{C}\times\mathcal{C} \to \mathcal{C}^{(2)}$ is the projection map. The affine variety in Equation~\eqref{kummer_middle} completes to a hypersurface in $\mathbb{P}(1,1,1,3)$ called the Shioda sextic \cite{MR2296439}. The proof of the following proposition was given in \cite{Clingher:2019aa}:
\begin{proposition}
\label{lem:Shioda}
The Shioda sextic determined by Equation~\eqref{kummer_middle} is birational to the Kummer surface $\operatorname{Kum}(\operatorname{Jac}\mathcal{C})$ associated with the Jacobian $\operatorname{Jac}(\mathcal{C})$ of the genus-two curve~$\mathcal{C}$ in Rosenhain normal form~\eqref{Eq:Rosenhain}. In particular, for the pairing of the Weierstrass points given by $(\lambda_1, \lambda_5), (\lambda_2, \lambda_3), (\lambda_4, \lambda_6)$ there is an isomorphism between Equation~\eqref{kummer_middle} and Equation~\eqref{Goepel-Quartic} such that
\begin{small}
\begin{gather}
\nonumber
 A =  2 \, \frac{\lambda_1+1}{\lambda_1-1}, \quad
 B =  2 \, \frac{\lambda_1\lambda_2+\lambda_1\lambda_3-2\lambda_2\lambda_3-2\lambda_1+\lambda_2+\lambda_3}{(\lambda_2-\lambda_3)(\lambda_1-1)},  \quad
 C  = 2 \, \frac{\lambda_3+\lambda_2}{\lambda_3-\lambda_2}, \\ 
 \label{KummerParameter4}
 D = 4 \, \frac{\lambda_1-\lambda_2 \lambda_3}{(\lambda_2 - \lambda_3) (\lambda_1-1)} \;.
\end{gather}
\end{small}%
\end{proposition}
\par The translation of the Jacobian $\mathcal{A}=\operatorname{Jac}(\mathcal{C})$ by a two-torsion point is an isomorphism of the Jacobian and maps the set of two-torsion points to itself. For any isotropic two-dimensional subspace $K \cong (\mathbb{Z}/2 \mathbb{Z})^2$ of $\mathcal{A}[2]$, also called called \emph{G\"opel group} in $\mathcal{A}[2]$, it is well-known that $\mathcal{A}'=\mathcal{A}/K$ is again a principally polarized abelian surface~\cite{MR2514037}*{Sec.~23}. Therefore, the isogeny $\Psi: \mathcal{A} \to \mathcal{A}'$ between principally polarized abelian surfaces has as its kernel the two-dimensional isotropic subspace $K$ of $\mathcal{A}[2]$. We call such an isogeny $\Psi$ a \emph{$(2,2)$-isogeny}, generalizing the isogeny in Remark~\ref{rem:duality}.  It turns out that there exists a complementary maximal isotropic subgroup in $\mathcal{A}[2]$, whose image in $\mathcal{A}'$ we denote by $K'$, such that $\mathcal{A}' /K' \cong \mathcal{A}$ and a corresponding dual $(2,2)$-isogeny $\Psi': \mathcal{A}' \to \mathcal{A}$. 
\par In the case $\mathcal{A}=\operatorname{Jac}(\mathcal{C})$ one may ask whether $\mathcal{A}' =\operatorname{Jac}(\mathcal{C}')$ for some other curve $\mathcal{C}'$ of genus two, and what the precise relationship between the moduli of $\mathcal{C}$ and $\mathcal{C}'$ is. The geometric moduli relationship between the two curves of genus two was found by Richelot \cite{MR1578135}; see \cite{MR970659}: if we choose for $\mathcal{C}$ a sextic equation $Y^2 = f_6(X,Z)$, then any factorization $f_6 = A\cdot B\cdot C$ into three degree-two polynomials $A, B, C$ defines a genus-two curve $\mathcal{C}'$ given by
\beq
\label{Richelot}
 \Delta \cdot Y^2 = [A,B] \, [A,C] \, [B,C] 
\eeq
where we have set $[A,B] = \partial A \, B - A \, \partial B$ with $\partial A$ the derivative of $A$ with respect to $X$ and $\Delta$ is the determinant of $A, B, C$ with respect to the basis $X^2, XZ, Z^2$. In \cite{Clingher:2019aa} a Richelot-isogeneous curve $\mathcal{C}'$ was obtained for the G\"opel group $K$ associated with pairing the roots according to $(\lambda_1,\lambda_5=1)$, $(\lambda_2,\lambda_3)$, $(\lambda_4=0,\lambda_6=\infty)$ and setting
\begin{small}
\begin{gather*}
[B,C]=x^2- \lambda_1,  [A,C]=x^2- \lambda_2\lambda_3, \\
[A,B]=(1+\lambda_1-\lambda_2-\lambda_3) \, x^2-2(\lambda_1-\lambda_2\lambda_3) \, x +\lambda_1\lambda_2+\lambda_1\lambda_3-\lambda_2\lambda_3-\lambda_1\lambda_2\lambda_3.
\end{gather*}
\end{small}%
A dual G\"opel group $K'$ is then obtained by setting $A'=[B,C]$, $B'=[A,C]$, $C'=[A,B]$ such that $\operatorname{Jac}(\mathcal{C}')/K' \cong \operatorname{Jac}(\mathcal{C})$ \cite{Clingher:2019aa}*{Sec.~3}. Moreover, when the isogeneous curve $\mathcal{C}'$ is written in the form
\beq
\label{Eq:Rosenhain_dual}
 \mathcal{C}': \quad y^2 = x \, \big(x-1\big) \, \big( x- \lambda'_1\big) \,  \big( x- \lambda'_2 \big) \,  \big( x- \lambda'_3 \big) \;,
\eeq 
with moduli $\Lambda'_1 = ( \lambda'_1 + \lambda'_2 \lambda'_3)/l'$,  $\Lambda'_2 = (\lambda'_2 + \lambda'_1 \lambda'_3)/l'$,  and  $\Lambda'_3 = (\lambda'_3 + \lambda'_1 \lambda'_2)/l'$ and $(l')^2 = \lambda'_1 \lambda'_2 \lambda'_3$, then a result of \cite{Clingher:2019aa}*{Sec.~3} proves that the moduli are related by
\begin{equation}
\label{relations_RosRoots}
 \begin{split}
 \begin{array}{rl}
   \Lambda_1 & = 2 \, \frac{2 \Lambda_1' - \Lambda_2'-\Lambda_3'}{\Lambda_2'-\Lambda_3'} \,,\\[0.6em]
   \Lambda_2 - \Lambda_1 & = - \frac{4(\Lambda_1'-\Lambda_2')(\Lambda_1'-\Lambda_3')}{(\Lambda_1'+2)(\Lambda_2'-\Lambda_3')} \,,\\[0.6em]
   \Lambda_3 - \Lambda_1 & = - \frac{4(\Lambda_1'-\Lambda_2')(\Lambda_1'-\Lambda_3')}{(\Lambda_1'-2)(\Lambda_2'-\Lambda_3')} \,,
  \end{array}
  & \qquad
 \begin{array}{rl}
   \Lambda'_1 & = 2 \, \frac{2 \Lambda_1 - \Lambda_2-\Lambda_3}{\Lambda_2 -\Lambda_3} \,,\\[0.6em]
   \Lambda'_2 - \Lambda'_1 & = - \frac{4(\Lambda_1-\Lambda_2)(\Lambda_1-\Lambda_3)}{(\Lambda_1+2)(\Lambda_2-\Lambda_3)} \,,\\[0.6em]
   \Lambda'_3 - \Lambda'_1 & = - \frac{4(\Lambda_1-\Lambda_2)(\Lambda_1-\Lambda_3)}{(\Lambda_1-2)(\Lambda_2-\Lambda_3)} \,.
  \end{array}   
  \end{split}
\end{equation}
\begin{remark}
\label{rem:symmetries2}
It was shown in \cite{Clingher:2019aa} that the two transformations changing signs of two coordinates in Equation~\eqref{eqn:reflections} generate the G\"opel group $K$ such that $\operatorname{Jac}(\mathcal{C}')/K' \cong \operatorname{Jac}(\mathcal{C})$ for the moduli related by Equations~\eqref{relations_RosRoots}.
\end{remark}
\subsection{Jacobian elliptic fibrations and Kummer sandwiches}
On the Kummer surface $\operatorname{Kum}(\mathcal{A})$ associated with a principally polarized abelian surface, there are always two sets of  sixteen $(-2)$-curves, called nodes and tropes, which are either the exceptional divisors corresponding to blow-up of the 16 two-torsion points or they arise from the embedding of the polarization divisor as symmetric theta divisors. These two sets of smooth rational curves have a rich symmetry, the so-called $16_6$-configuration  where each node intersects exactly six tropes and vice versa \cite{MR1097176}.  Using curves in the $16_6$-configuration,  one can find all elliptic fibrations since all irreducible components of a reducible fiber in an elliptic fibration are $(-2)$-curves \cite{MR0184257} and the Picard rank is seventeen. All inequivalent elliptic fibrations were determined explicitly by Kumar in \cite{MR3263663}. In particular, Kumar computed elliptic parameters and  Weierstrass equations for all twenty five different fibrations that appear, and analyzed the reducible fibers and  Mordell-Weil lattices. 
\par It was shown in \cites{Clingher:2017aa, MR4015343} that the Shioda sextic in Equation~\eqref{kummer_middle} defines a Jacobian elliptic fibration on $\mathcal{S}'=\operatorname{Kum}(\operatorname{Jac} \mathcal{C}')$ associated with the Jacobian of the genus-two $\mathcal{C}'$ in Equation~\eqref{Eq:Rosenhain_dual} given by the Weierstrass equation
\beq
\label{kummer_middle_ell_p_W}
\begin{split}
y^2  = x & \big( x-  u \left( u^2 - u \, \Lambda'_3  +1 \right)   \left( \Lambda_1'-\Lambda_2' \right) \big) \, \big( x- u\left( u^2 - u \, \Lambda'_2  +1 \right)    \left( \Lambda_1'-\Lambda_3' \right) \big) \,,
\end{split}
\eeq
such that the holomorphic two-form in Equation~\eqref{eqn:2form} coincides with $\omega_{\mathcal{S}'} = du \wedge dx/y$. We have the following:
\begin{lemma}
\label{lem:FibShioda}
Equation~\eqref{kummer_middle_ell_p_W} determines an elliptic fibration with section on the Kummer surface $\operatorname{Kum}(\operatorname{Jac} \mathcal{C}')$. Generically, the Weierstrass model has two singular fibers of Kodaira-type $I_0^*$ at $u=0, \infty$, six singular fibers of type $I_2$ located over a collection of base points invariant under $u \mapsto 1/u$, and the Mordell-Weil group $(\mathbb{Z}/2\mathbb{Z})^2 \oplus \langle 1 \rangle$.
\end{lemma}
\begin{proof}
One first identifies the collection of singular fibers as in the proof of Lemma~\ref{lem:J1}. Comparison with the results in \cite{MR3263663} then determines the Mordell-Weil group.
\end{proof}
\par The Jacobian elliptic Kummer surface $\mathcal{S}'=\operatorname{Kum}(\operatorname{Jac} \mathcal{C}')$ in Lemma~\ref{lem:FibShioda} admits an additional two-torsion section, namely  $T': (x,y)=(0,0)$, defining a Van Geemen-Sarti involution $k' \in \operatorname{Aut}(\mathcal{S}')$. It is straight forward to show that the minimal resolution of $\mathcal{S}'/\langle k' \rangle$ is a Jacobian elliptic K3 surface $\mathcal{Y}$ given by the Weierstrass equation
\beq
\label{kummer_middle_ell_upper}
\begin{split}
 Y^2=&\, X^3 +2 \, u^2\Big( (2 \Lambda'_1 -\Lambda'_2 -\Lambda'_3) (u+u^{-1}) +(2 \Lambda'_2\Lambda'_3 - \Lambda'_1\Lambda'_2-\Lambda'_1\Lambda'_3 )\Big)  X^2 \\
 & + u^4 (\Lambda'_2 -\Lambda'_3)^2 \big( (u+u^{-1}) - \Lambda'_1 \big)^2 X \,,
\end{split} 
\eeq
equipped with the holomorphic two-form $\omega_\mathcal{Y} = du \wedge dX/Y$. Obviously, the K3 surface $\mathcal{Y}$ also admits a second (commuting) Nikulin involution
\beq
\label{eqn:involution2b_top}
 \jmath': \;  (u, X, Y) \mapsto \left( \frac{1}{u}, \frac{X}{u^4}, - \frac{Y}{u^6} \right)\,.
\eeq
We have the following:
\begin{lemma}
\label{lem:FibShioda_upper}
Equation~\eqref{kummer_middle_ell_upper} determines an elliptic fibration with section on the K3 surface $\mathcal{Y}$. Generically, the Weierstrass model has two singular fibers of Kodaira-type $I_0^*$ at $u=0, \infty$, two singular fibers of type $I_4$ and four singular fibers of type $I_1$  located over a collection of base points invariant under $u \mapsto 1/u$, and the Mordell-Weil group $\mathbb{Z}/2\mathbb{Z} \oplus \langle 1 \rangle$.
\end{lemma}
\begin{proof}
The proof is similar to the one of Lemma~\ref{lem:FibShioda}.
\end{proof}
We have the following result analogous to Proposition~\ref{prop2}:
\begin{proposition}
\label{prop2_17}
The K3 surface $\mathcal{Y}$ (given by Equation~\eqref{kummer_middle_ell_p_W}) and the Kummer surface $\mathcal{S}'$ (given by Equation~\eqref{kummer_middle_ell_upper}) admit dual Van Geemen-Sarti involution $k$ and $k'$ associated with fiberwise translations by the order-two section $T: (X,Y)=(0,0)$ and $T': (x,y)=(0,0)$, respectively, and a pair of dual geometric two-isogenies
\begin{equation}
\label{isog_middle_17}
 \xymatrix 
{ \mathcal{Y} \ar @(dl,ul) _{k}
\ar @/_0.5pc/ @{-->} _{\varphi} [rr] &
& \mathcal{S}' \ar @(dr,ur) ^{k'}
\ar @/_0.5pc/ @{-->} _{\varphi'} [ll] \\
} 
\end{equation}
such that $\omega_\mathcal{Y} = \varphi^* \omega_{\mathcal{S}'}$ and $2 \, \omega_{\mathcal{S}'} =  (\varphi')^* \omega_{\mathcal{Y}}$ for the holomorphic two-forms $\omega_{\mathcal{S}'} = du \wedge dx/y$ and $\omega_\mathcal{Y}=du \wedge dX/Y$ on  $\mathcal{S}'$ and $\mathcal{Y}$, respectively.
\end{proposition}
\begin{proof}
The proof follows the one for Proposition~\ref{prop2} applied to the Jacobian elliptic fibrations in Equation~\eqref{kummer_middle_ell_p_W} and Equation~\eqref{kummer_middle_ell_upper}.
\end{proof}
We also have the following result from \cite{MR3263663}:
\begin{lemma}
\label{lem:fibKUM_JAC_C}
The Kummer surface $\mathcal{S}=\operatorname{Kum}(\operatorname{Jac} \mathcal{C})$ admits:
\begin{enumerate}
\item a Jacobian elliptic fibration whose Weierstrass model has four singular fibers of Kodaira-type $I_4$, eight singular fibers of type $I_1$ located over a collection of base points invariant under $t \mapsto -t$ and $t \mapsto 1/t$, and the Mordell-Weil group $\mathbb{Z}^3 \oplus \mathbb{Z}/2\mathbb{Z}$,
\item  a Jacobian elliptic fibration whose Weierstrass model has one singular fiber of Kodaira-type $I_4$, two singular fibers of type $I_1$, three singular fibers of type $I_0^*$, and the Mordell-Weil group $\mathbb{Z}/2\mathbb{Z}$.
\end{enumerate}
\end{lemma} 
\par As shown in \cite{Clingher:2017aa}, a Weierstrass equation for fibration~(1) in Lemma~\ref{lem:fibKUM_JAC_C} is
\beq
\label{kummer_ell_upper_left}
\begin{split}
 y^2=&\, x^3 +2 \, t^2\Big( (2 \Lambda'_1 -\Lambda'_2 -\Lambda'_3) (t^2+t^{-2}) +(2 \Lambda'_2\Lambda'_3 - \Lambda'_1\Lambda'_2-\Lambda'_1\Lambda'_3 )\Big)  x^2 \\
 & + t^4 (\Lambda'_2 -\Lambda'_3)^2 \big( (t^2+t^{-2}) - \Lambda'_1 \big)^2 x \,,
\end{split} 
\eeq
equipped with the holomorphic two-form $\omega_\mathcal{S} = dt \wedge dx/y$. It admits the two commuting Nikulin involutions $\imath, \jmath$ given by
\beq
\label{eqn:involutions17}
 \imath: \, (t, x, y) \mapsto ( -t, x,-y) \,, \qquad  \jmath: \,  (t, x, y) \mapsto \left( \frac{1}{t}, \frac{x}{t^4}, - \frac{y}{t^6} \right) \,.
\eeq 
Similarly, a Weierstrass equation for fibration~(2) in Lemma~\ref{lem:fibKUM_JAC_C} is
\beq
\label{kummer_ell_upper_right}
\begin{split}
 y^2=&\, x^3 +2 \,(v^2-4) \Big( (2 \Lambda'_1 -\Lambda'_2 -\Lambda'_3)v +(2 \Lambda'_2\Lambda'_3 - \Lambda'_1\Lambda'_2-\Lambda'_1\Lambda'_3 ) \Big)  x^2 \\
 & +(v^2-4)^2 (\Lambda'_2 -\Lambda'_3)^2 (v- \Lambda'_1)^2 x \,,
\end{split} 
\eeq
equipped with the holomorphic two-form $\omega_\mathcal{S} = dv \wedge dx/y$.  Equation~\eqref{kummer_ell_upper_right} also admits an additional two-torsion section, namely  $T: (x,y)=(0,0)$, defining a Van Geemen-Sarti involution $k \in \operatorname{Aut}(\mathcal{S})$.  For convenience, we have expressed the coefficients of the fibration in terms of the moduli of the $(2,2)$-isogeneous curve $\mathcal{C}'$, rather than the moduli of  the curve $\mathcal{C}$. Because the coefficients of the Weierstrass equations only depend on $(\Lambda'_1, \Lambda'_2, \Lambda'_3)$, they can easily be expressed in terms of the moduli of the curve $\mathcal{C}$ using Equation~\eqref{relations_RosRoots}.
\par Thus, the Kummer sandwich theorem~\ref{prop1} generalizes as follows:
\begin{proposition}\label{prop1_17}
The Kummer surface $\mathcal{S}= \operatorname{Kum}(\operatorname{Jac} \mathcal{C})$ admits two commuting Nikulin involutions $\imath, \jmath$ (given by Equations~\eqref{eqn:involutions17}) such that the quotient surface $\mathcal{S}/\langle \imath \rangle$ is birational to $\mathcal{Y}$ in Equation~\eqref{kummer_middle_ell_upper}, and the quotient surface  $\mathcal{S}/\langle \imath, \jmath \rangle$ is birational to $\mathcal{S}$ itself. Thus, $\mathcal{S}$ dominates and is dominated by $\mathcal{Y}$ via the two rational maps of degree two
\beq
 \psi: \; \mathcal{S} \dasharrow \mathcal{Y} \,, \qquad \phi:  \mathcal{Y} \dasharrow \mathcal{S} \,,
 \eeq
given by Equations~\eqref{eqn:psi} and~\eqref{eqn:phi}, such that $2 \, \omega_\mathcal{S} = \psi^* \omega_\mathcal{Y} $ and $\omega_\mathcal{Y} = \phi^* \omega_{\mathcal{S}}$ for the holomorphic two-forms $\omega_\mathcal{Y}$ and $\omega_\mathcal{S}$ on $\mathcal{Y}$ and $\mathcal{S}$, respectively.
\end{proposition}
\begin{proof}
The proof is the same as the proof for Proposition~\ref{prop1_17}.
\end{proof}
\par The involution $\jmath$ induces Nikulin involutions $\jmath' \in \operatorname{Aut}(\mathcal{S}')$, in exactly the same way it did in Picard rank 18. We obtain a new K3 surface $\mathcal{Y}'$ as the minimal resolution of $\mathcal{S}'/\langle \jmath' \rangle$. The quotient map induces the degree-two rational map $\phi': \mathcal{S}' \dasharrow \mathcal{Y}'$ (given by Equation~\eqref{eqn:phi_prime}) such that $(\phi')^* \omega_{\mathcal{Y}'} =  \omega_{\mathcal{S}'}$ for the holomorphic two-form $\omega_{\mathcal{Y}'} = dv \wedge dX/Y$ on $\mathcal{Y}'_{\lambda_1, \lambda_2}$, where a Weierstrass equation for $\mathcal{Y}'$ is given by
\beq
\label{eqn:Zfib17}
\begin{split}
Y^2 = & \, X^3 - (v^2-4) \Big( (2 \Lambda'_1 -\Lambda'_2 -\Lambda'_3)v +(2 \Lambda'_2\Lambda'_3 - \Lambda'_1\Lambda'_2-\Lambda'_1\Lambda'_3 ) \Big)  X^2 \\
& + (v^2-4)^2  ( \Lambda'_1 -\Lambda'_2) ( \Lambda'_1 -\Lambda'_3)  (v -\Lambda'_2)  (v -\Lambda'_3) \,X \,.
\end{split}
\eeq
We have the following:
\begin{lemma}
\label{lem:Zfib17}
Equation~\eqref{eqn:Zfib17} defines an elliptic fibration with section on the K3 surface $\mathcal{Y}'$. Generically, the Weierstrass model has three singular fibers of Kodaira-type $I_0^*$, three singular fibers of type $I_2$, and the Mordell-Weil group $(\mathbb{Z}/2\mathbb{Z})^2$.
\end{lemma}
\begin{proof}
The proof is analogous to the proof of Lemma~\ref{lem:J7}.
\end{proof}
\par Corollary~\ref{prop3} then generalizes as well. We have:
\begin{corollary}
\label{prop3_17}
The K3 surface $\mathcal{Y}'$ and the Kummer surface $\mathcal{S}= \operatorname{Kum}(\operatorname{Jac} \mathcal{C})$ admit dual Van Geemen-Sarti involution $k$ and $k'$ associated with fiberwise translations by the order-two section $T': (X,Y)=(0,0)$ and $T: (x,y)=(0,0)$, respectively, and a pair of dual geometric two-isogenies
\begin{equation}
\label{isog_right_17}
 \xymatrix 
{ \mathcal{Y}' \ar @(dl,ul) _{k'}
\ar @/_0.5pc/ @{-->} _{\chi'} [rr] &
& \mathcal{S} \ar @(dr,ur) ^{k}
\ar @/_0.5pc/ @{-->} _{\chi} [ll] \\
} 
\end{equation}
such that $\omega_\mathcal{S} = \chi^* \omega_{\mathcal{Y}'}$ and $2 \, \omega_{\mathcal{Y}'} =  (\chi')^* \omega_{\mathcal{S}}$ for the holomorphic two-forms $\omega_{\mathcal{Y}'} = dv \wedge dx/y$ and $\omega_\mathcal{S}=du \wedge dX/Y$ on $\mathcal{S}_{\lambda_1, \lambda_2}$ and $\mathcal{Y}'_{\lambda_1, \lambda_2}$, respectively.
\end{corollary}
\begin{proof}
The proof is analogous to the proof of Theorem~\ref{prop2}.
\end{proof}
\begin{figure}[ht]
\centerline{
\xymatrix{
*+[F--]{\mathcal{S} = \operatorname{Kum}(\operatorname{Jac}\mathcal{C})}  \ar@{-->}[r]^{\psi}   \ar@{->}[rd]^{\Psi}
	& *+[F-,]{\mathcal{Y}} \ar@{-->}[r]^{\phi}  \ar@{-->}[d]^{\varphi} 
	& *+[F--]{\mathcal{S} = \operatorname{Kum}(\operatorname{Jac}\mathcal{C})}  \ar@{-->}[d]^{\chi} \\
	& *+[F--]{\mathcal{S}' = \operatorname{Kum}(\operatorname{Jac}\mathcal{C})}  \ar@{-->}[r]^{\phi'}  \ar@{-->}[u]<1ex>^{\varphi'}  \ar@{->}[ru]^{\Psi'}
	& *+[F-,]{\mathcal{Y}'} \ar@{-->}[u]<1ex>^{\chi'} 
}}
\caption{\label{fig2}}
\end{figure}
\par We summarize the results of Propositions~\ref{prop1_17} and~\ref{prop2_17}, and Corollary~\ref{prop3_17} in Diagram~\ref{fig2}. Here, $\mathcal{C}$ is the genus-two curve in Equation~\eqref{Eq:Rosenhain}, $\mathcal{C}'$ is the $(2,2)$-isogeneous genus-two curve in Equation~\eqref{Eq:Rosenhain_dual}, and $\Psi$ and $\Psi'$ are the rational maps induced by the $(2,2)$-isogenies between them. The moduli of the genus-two curve $\mathcal{C}$ are the Rosenhain roots $(\lambda_1, \lambda_2, \lambda_3)$. From those we obtain $\Lambda_1 = ( \lambda_1 + \lambda_2 \lambda_3)/l$,  $\Lambda_2 = (\lambda_2 + \lambda_1 \lambda_3)/l$,  and  $\Lambda_3 = (\lambda_3 + \lambda_1 \lambda_2)/l$ on the double cover $\widetilde{\mathcal{M}}_2$ of the moduli space given by $l^2 = \lambda_1 \lambda_2 \lambda_3$. The moduli $\Lambda'_1, \Lambda'_2, \Lambda'_3$ are obtained in the same way from the Rosenhain roots $(\lambda'_1, \lambda'_2, \lambda'_3)$ of $\mathcal{C}'$, and related via Equations~\eqref{relations_RosRoots}.
\par We now state our first main result:
\begin{theorem}
\label{prop6}
The Kummer surface $\mathcal{S}= \operatorname{Kum}(\operatorname{Jac} \mathcal{C})$ admits three commuting Nikulin involutions $\imath, \jmath, k$ such that the quotient surface $\mathcal{S}/\langle \imath \rangle$ is birational to $\mathcal{Y}$ in Equation~\eqref{kummer_middle_ell_upper}, $\mathcal{S}/\langle \imath, \jmath \rangle$ is birational to $\mathcal{S}$, and  $\mathcal{S}/\langle \imath, \jmath, k \rangle$ is birational to $\mathcal{Y}'$ in Equation~\eqref{eqn:Zfib17}. In particular, there are rational maps of degree two
\beq
 \operatorname{Kum}(\operatorname{Jac} \mathcal{C}) \overset{\psi}{\dashrightarrow} \mathcal{Y} 
 \overset{\phi}{\dashrightarrow} \operatorname{Kum}(\operatorname{Jac} \mathcal{C}) \overset{\chi}{\dashrightarrow} 
 \mathcal{Y}' \,,
\eeq 
and
\beq
\label{eqn:sandwich2}
 \operatorname{Kum}(\operatorname{Jac} \mathcal{C}) \overset{\psi}{\dashrightarrow} \mathcal{Y} 
 \overset{\varphi}{\dashrightarrow} \operatorname{Kum}(\operatorname{Jac} \mathcal{C}') \overset{\phi'}{\dashrightarrow} 
 \mathcal{Y}' \,,
\eeq 
such that $\omega_\mathcal{Y} = \phi^*\chi^* \omega_{\mathcal{Y}'}= \varphi^*\phi'^* \omega_{\mathcal{Y}'}$ for the holomorphic two-forms $\omega_\mathcal{Y} = du \wedge dX/Y$ and $\omega_{\mathcal{Y}'}=dv \wedge dX/Y$. The K3 surface $\mathcal{Y}'$ is birational to the Legendre pencil
\beq
\label{eqn:Legendre19b}
  \tilde{y}^2 = - \frac{16(\Lambda_1-\Lambda_3)(\Lambda_1-\Lambda_2)}{(\Lambda_2-\Lambda_3)^2(\Lambda_1^2-4)} \, \big(\tilde{x}^2-4\big) \big(\tilde{x}-\tilde{v}\big) \big(\tilde{v}- \Lambda_1\big) \big(\tilde{v}- \Lambda_2\big) \big(\tilde{v} - \Lambda_3\big)  \,,
\eeq
equipped with the holomorphic two-form $d\tilde{v} \wedge d\tilde{x}/\tilde{y}$ or, equivalently, given by
\beq
\label{eqn:Legendre19a}
  y^2 = - \big(x^2-4\big) \big(x -w\big) \big( w - \Lambda'_1\big) \big( w - \Lambda'_2\big) \big( w - \Lambda'_3\big) \,,
\eeq
equipped with the holomorphic two-form $dw \wedge dx/y$.
\end{theorem}
\begin{proof}
The application of Propositions~\ref{prop1_17} and~\ref{prop2_17}, and Corollary~\ref{prop3_17} and Diagram~\ref{fig2} prove the first part of the theorem. To obtain Equation~\eqref{eqn:Legendre19a} from Equation~\eqref{eqn:Zfib17}, one interchanges the roles of base and fiber and uses the transformation $x=v$ and
\beq
 w = \frac{(v^2-4)(v- \Lambda'_1) (v- \Lambda'_2) (v- \Lambda'_3)}{X-(v^2-4)(v- \Lambda'_2) (v- \Lambda'_3)} + v \,.
\eeq 
To obtain Equation~\eqref{eqn:Legendre19b} from Equation~\eqref{eqn:Zfib17}, one uses a fractional liner transformation with
\beq
 v = \frac{2(2\Lambda_1-\Lambda_2-\Lambda_3)w-2(\Lambda_1\Lambda_2+\Lambda_1\Lambda_3-2\Lambda_2\Lambda_3)}{(w- \Lambda_1) (\Lambda_2- \Lambda_3)} \,. \qedhere
\eeq 
\end{proof}
We make the following critical remarks:
\begin{remark}
Diagram~\ref{fig2} and Remark~\ref{rem:symmetries2} imply that the G\"opel group $K$ such that $\operatorname{Jac}(\mathcal{C}')/K' \cong \operatorname{Jac}(\mathcal{C})$ is in fact generated by the Nikulin involution $i$ and the Van Geemen-Sarti involution $k$ on the Jacobian elliptic fibration~\eqref{kummer_ell_upper_left} on $\operatorname{Kum}(\operatorname{Jac} \mathcal{C})$. The precise relationship between the action of G\"opel groups and the Jacobian elliptic fibrations was discussed in \cite{Clingher:2019ab}.
\end{remark}

\begin{remark}
The K3 surfaces $\mathcal{Y}'$ in Theorem~\ref{prop6} are double covers of the Hirzebruch surface $\mathbb{F}_0=\mathbb{P}^1\times\mathbb{P}^1$ branched along a curve of type $(4,4)$, i.e., along a section in the line bundle $\mathcal{O}_{\mathbb{F}_0}(4,4)$. Every such cover has two elliptic fibrations corresponding to the two rulings of the quadric $\mathbb{F}_0$ coming from the projections $\pi_i: \mathbb{F}_0 \to \mathbb{P}^1$ for $i=1,2$. A fibration with two fibers of type $I_0^*$ corresponds to double covers of $\mathbb{F}_0$ branched along curves of the form $F_1+F_2 +S$ where $F_1, F_2$ are fibers of $\pi_1$ and $S$ is a section of $\mathcal{O}_{\mathbb{F}_0}(2,4)$. A second elliptic fibration arises from the projection $\pi_2$, and in this case has the same singular fibers. This second fibration arises in a simple geometric manner, roughly speaking, by interchanging the roles of base and fiber coordinates for the first fibration; see details in the proof of Theorem~\ref{prop6}.
\end{remark}
\begin{remark}
\label{rem:critical}
Using the presentation of $\mathcal{Y}'$ as twisted Legendre pencils in Equation~\eqref{eqn:Legendre19a} and~\eqref{eqn:Legendre19a}, any  period integral $f$ of the holomorphic two-form $\omega_{\mathcal{Y}'}$ is a function of $\Lambda'_1,\Lambda'_2,\Lambda'_3$ and $\Lambda_1,\Lambda_2,\Lambda_3$, respectively, such that
\beq
\label{eqn:period}
 f = f(\Lambda'_1,\Lambda'_2,\Lambda'_3) = \nu(\Lambda_1,\Lambda_2,\Lambda_3)^{\frac{1}{2}} \; f(\Lambda_1,\Lambda_2,\Lambda_3) 
\eeq 
with
\beq
  \nu(\Lambda_1,\Lambda_2,\Lambda_3) = \frac{16(\Lambda_1-\Lambda_3)(\Lambda_1-\Lambda_2)}{(\Lambda_2-\Lambda_3)^2(\Lambda_1^2-4)}   \,.
\eeq
Replacing $\operatorname{Jac}(\mathcal{C})$ by the $(2,2)$-isogeneous abelian surface $\operatorname{Jac}(\mathcal{C}')$ in the Kummer surface $\mathcal{S}$ in Theorem~\ref{prop6} amounts to interchanging the two fibrations in Equation~\eqref{eqn:Legendre19a} and Equation~\eqref{eqn:Legendre19b} up to a twist. Accordingly, one checks that for $(\Lambda'_1,\Lambda'_2,\Lambda'_3) \mapsto (\Lambda_1,\Lambda_2,\Lambda_3)$, given by Equations~\eqref{relations_RosRoots}, we have $\nu(\Lambda_1,\Lambda_2,\Lambda_3) \mapsto 1/\nu(\Lambda'_1,\Lambda'_2,\Lambda'_3)$ ensuring the equivariance of Equation~\eqref{eqn:period}.
\end{remark}
\par Finally, we discuss how the limit is attained when the Picard number of the constructed Jacobian elliptic K3 surfaces increases to $18$. This is based on techniques developed in \cites{MR3731039, MR3712162,MR3366121}. We have the following:
\begin{proposition}
A pencil of genus-two curves $\mathcal{C}'_{\epsilon'}$ with  $\lambda'_1 \to \lambda'_1$, $\lambda'_2\to \lambda'_2 (\epsilon')^2$, $\lambda'_3 \to (\epsilon')^2$ and $0 < |\epsilon|<1$ in Equation~\eqref{Eq:Rosenhain_dual} is of (parabolic) type $[I_{4-0-0}]$ in the Namikawa-Ueno classification~\cite{MR369362}.  Using the moduli of the pencil in the limit $\epsilon' \to 0$, the elliptic fibrations in Equations~\eqref{kummer_ell_upper_left},  \eqref{kummer_middle_ell_upper}, \eqref{kummer_ell_upper_right}, \eqref{kummer_middle_ell_p_W}, \eqref{eqn:Zfib17}  coincide with the elliptic fibrations in Picard rank 18 in Equations~\eqref{eqn:J1}, \eqref{eqn:B12}, \eqref{eqn:J7}, \eqref{eqn:J6dual}, \eqref{eqn:Zfib},  respectively, and Diagram~\ref{fig2} coincides with Diagram~\ref{fig1}.
\end{proposition}
\begin{proof}
It is easy to show that a pencil $\mathcal{C}'_{\epsilon'}$ with $\lambda'_1=(k'_1)^2$, $\lambda'_2=(k'_2 \epsilon')^2$, $\lambda'_3=(\epsilon')^2$  can be written in the form
\beq 
 y^2 = (x^3 + \alpha x + \beta) ( (x+\gamma)^2 + (\epsilon')^4) \,.
\eeq
This proves that the pencil is of (parabolic) type $[I_{4-0-0}]$ in the Namikawa-Ueno classification. We obtain
\beq
 \Lambda'_1 = \frac{k'_2 (\epsilon')^2}{k'_1} +  \frac{k'_1}{ k'_2 (\epsilon')^2}\,, \quad \Lambda'_2 = \frac{k'_1}{k'_2} + \frac{k'_2}{k'_1} \,, \quad \Lambda'_3 = k'_1 k'_2 + \frac{1}{k'_1 k'_2} \,,
\eeq 
where we consider $(k'_1)^2$ and $(k'_2)^2$ the moduli of two elliptic curves $\mathcal{E}'_1$ and $\mathcal{E}'_2$ in Equation~\eqref{EllC1_dual}.  To obtain the parameters in terms of two two-isogeneous elliptic curves $\mathcal{E}_1$ and $\mathcal{E}_2$ in Equation~\eqref{eqn:EC}, we set $k'_1 = (1-k_1)/(1+k_1)$ and  $k'_2 = (1-k_2)/(1+k_2)$ and also set $\epsilon'=2 \, \epsilon/[(k_1+1)(k_2-1)]$; see the proof of Lemma~\ref{lem:J6moduli}. Then, the limit of the constructed elliptic fibrations in Picard rank 17 for $\epsilon \to 0$ returns the corresponding elliptic fibrations in Picard rank 18 from Section~\ref{ssec:Kummer18} (up to a rescaling of the affines variables of the elliptic fiber, say $(x,y) \mapsto (x/\epsilon^2, y/\epsilon^3)$, and Diagram~\ref{fig2} becomes Diagram~\ref{fig1} for $\epsilon \to 0$.
\end{proof}
\subsection{Rational point-count on the twisted Legendre pencil}
\label{sec:PointCount}
As a reminder, we have constructed the three-parameter families of K3 surfaces $\mathcal{S}'$ and $\mathcal{Y}'$  shown in Diagram~\ref{fig2}. The K3 surface $\mathcal{Y}'$  is dominated by the Kummer surfaces $\mathcal{S}= \operatorname{Kum}(\operatorname{Jac} \mathcal{C})$ and $\mathcal{S}= \operatorname{Kum}(\operatorname{Jac} \mathcal{C}')$. Here, the abelian surface $\operatorname{Jac} \mathcal{C}$ is the Jacobian of the genus-two curve $\mathcal{C}$ with Rosenhain roots $(\lambda_1, \lambda_2, \lambda_3)$. From those moduli we obtain $\Lambda_1 = ( \lambda_1 + \lambda_2 \lambda_3)/l$,  $\Lambda_2 = (\lambda_2 + \lambda_1 \lambda_3)/l$,  and  $\Lambda_3 = (\lambda_3 + \lambda_1 \lambda_2)/l$ on the double cover $\widetilde{\mathcal{M}}_2$ of the moduli space given by $l^2 = \lambda_1 \lambda_2 \lambda_3$. The moduli $\Lambda'_1, \Lambda'_2, \Lambda'_3$ are obtained in the same way from the Rosenhain roots $(\lambda'_1, \lambda'_2, \lambda'_3)$ for an isogeneous curve $\mathcal{C}'$ related via Equations~\eqref{relations_RosRoots} such that $\mathcal{S}'= \operatorname{Kum}(\operatorname{Jac} \mathcal{C}')$. That is, the abelian surfaces are related by a $(2,2)$-isogeny $\Psi: \operatorname{Jac} \mathcal{C} \to \operatorname{Jac} \mathcal{C}'$ and its dual $(2,2)$-isogeny $\Psi': \operatorname{Jac} \mathcal{C}' \to \operatorname{Jac} \mathcal{C}$.
\par Denote by $|\mathcal{S}'|_p$ and $|\mathcal{Y}'|^{(1)}_p$ or $|\mathcal{Y}'|^{(2)}_p$ the number of rational points of the affine part of $\mathcal{S}'$ and $\mathcal{Y}'$, defined using the Jacobian elliptic fibrations in Equation~\eqref{kummer_middle_ell_upper} and Equation~\eqref{eqn:Legendre19b} or~\eqref{eqn:Legendre19a}, respectively, over the finite field $\mathbb{F}_p$. We have the following:
\begin{proposition} 
We have:
\begin{enumerate}
\item $|\mathcal{S}'|_p = 2\, |\mathcal{Y}'|^{(1)}_p$,
\item $|\mathcal{Y}'|^{(2)}_p =  \nu(\Lambda_1,\Lambda_2,\Lambda_3)^{\frac{1}{2}} \; |\mathcal{Y}'|^{(1)}_p$.
\end{enumerate}
\end{proposition}
\begin{proof}
The proof of (1) follows  by the application of Manin's principle and Theorem~\ref{prop6} and the relation between the elliptic fibrations in Equation~\eqref{kummer_middle_ell_upper} and  Equation~\eqref{eqn:Legendre19b}. The detailed argument for the relation between rational-point counting functions and period integrals can be found in the proof of \cite{MR3992148}*{Thm.~5.1}. Statement (2) follows from Remark~\ref{rem:critical}.
\end{proof}
We now state our second main result:
\begin{theorem}
\label{thm:main2}
Let $\Lambda'_1,\Lambda'_2,\Lambda'_3 \in\mathbb{Q}$. The following identity holds
\beq
\begin{split}
|\mathcal{Y}'|^{(2)}_p 
\equiv  1+(-1)^{\frac{p-1}{2}} 
\sum_{\ell=0}^{\frac{p-1}{2}} \, 2^\ell \kern-1em \sum_{\substack{s+t+\ell=p-1 \\ 0\le s,t \le\frac{p-1}{2}}}
\sum_{\substack{i+j+k+\ell=p-1 \\ 0\le i,j,k\le \frac{p-1}{2}}} 
&C^{\frac{p-1}{2}}_s C^{\frac{p-1}{2}}_t C^{\frac{p-1}{2}}_\ell
C^{\frac{p-1}{2}}_i C^{\frac{p-1}{2}}_j C^{\frac{p-1}{2}}_k \\
&\kern-1em \times \,  (-1)^t
({\Lambda'}_1)^i ({\Lambda'}_2)^j ({\Lambda'}_3)^k \mod p \,,
\end{split}
\eeq
where $C^n_k = \frac{n!}{k!\,(n-k)!}$ is the binomial coefficient of $n$ chose $k$.
\end{theorem}
\begin{proof}
In Theorem~\ref{prop6} the K3 surface $\mathcal{Y}'$ was shown to be birational to the twisted Legendre pencil
\beq
  y^2 = - \big(x^2-4\big) \big(x -w\big) \big( w - \Lambda'_1\big) \big( w - \Lambda'_2\big) \big( w - \Lambda'_3\big) \,,
\eeq
which we used to define the counting function $|\mathcal{Y}'|^{(2)}_p$. For simplicity, we set $a=\Lambda'_1,b=\Lambda'_2,b=\Lambda'_3 \in\mathbb{Q}$. Using standard techniques from \cite{MR3992148}, we calculate $|\mathcal{Y}'|^{(2)}_p$ as follows:
\beqn
\begin{split}
|\mathcal{Y}'|^{(2)}_p &= \sum_{x, t\in\mathbb{F}_p} \big( -(x+2)(x-2)(x-w)(w-a)(w-b)(w-c) \big)^{\frac{p-1}{2}} \\
&= (-1)^{\frac{p-1}{2}} \sum_{w\in\mathbb{F}_p} \Big( \underbrace{\sum_{x\in\mathbb{F}_p} 
\big((x+2)(x-2)(x-w) \big)^{\frac{p-1}{2}}}_{(a)} \Big)
\big( (w-a)(w-b)(w-c) \big)^{\frac{p-1}{2}}.
\end{split}
\eeqn
The contribution $(a)$ can be computed using the formula
\beqn
\begin{split}
&\sum_{x\in\mathbb{F}_p} \big( (x+2)(x-2)(x-w) \big)^{\frac{p-1}{2}} \\
&= \sum_{x\in\mathbb{F}_p} \sum_{0\le s,t,\ell\le\frac{p-1}{2}}
C^{\frac{p-1}{2}}_s C^{\frac{p-1}{2}}_t C^{\frac{p-1}{2}}_\ell x^{s+t+\ell} 2^{(p-1)-(s+t)}
(-1)^{(p-1)-(t+\ell)} w^{\frac{p-1}{2}-\ell} \\
&\equiv - \sum_{\substack{s+t+\ell=p-1 \\ 0\le s,t,\ell\le\frac{p-1}{2}}}
C^{\frac{p-1}{2}}_s C^{\frac{p-1}{2}}_t C^{\frac{p-1}{2}}_\ell 
2^\ell (-1)^s w^{\frac{p-1}{2}-\ell}.
\end{split}
\eeqn
By letting $\ell_1=\frac{p-1}{2}-\ell$, $(a)$ can be re-written as
\beqn
- \sum_{\ell_1=0}^{\frac{p-1}{2}} 
\underbrace{\Bigg( \sum_{\substack{s+t=\frac{p-1}{2}+\ell_1 \\ 0\le s,t \le\frac{p-1}{2}}}
C^{\frac{p-1}{2}}_s C^{\frac{p-1}{2}}_t C^{\frac{p-1}{2}}_{\ell_1}
2^{\frac{p-1}{2}-\ell_1} (-1)^s \Bigg)}_{=A_{\ell_1}} 
w^{\ell_1} \,.
\eeqn
For more details, we refer the reader to the analogous proof of \cite{MR3992148}*{Prop.~4.6}. Hence, we obtain
\beqn
|\mathcal{Y}'|^{(2)}_p \equiv -(-1)^{\frac{p-1}{2}} \sum_{w\in\mathbb{F}_p} 
\sum_{\ell_1=0}^{\frac{p-1}{2}} A_{\ell_1} w^{\ell_1}
\big( (w-a)(w-b)(w-c) \big)^{\frac{p-1}{2}} \,.
\eeqn
We now sum over $w\in\mathbb{F}_p$ to further simplify the summation. Consider the identity
\beqn
\begin{split}
&\sum_{w\in\mathbb{F}_p} w^{\ell_1} \big( (w-a)(w-b)(w-c) \big)^{\frac{p-1}{2}} \\
&= - \kern-1em \sum_{0\le i,j,k \le\frac{p-1}{2}}  
C^{\frac{p-1}{2}}_i C^{\frac{p-1}{2}}_j C^{\frac{p-1}{2}}_k w^{i+j+k+\ell_1}
(-a)^{\frac{p-1}{2}-i} (-b)^{\frac{p-1}{2}-j} (-c)^{\frac{p-1}{2}-k} \,.
\end{split}
\eeqn
If $i+j+k+\ell_1 = 2(p-1)$, then we have $i=j=k=\ell_1=\frac{p-1}{2}$, which implies that the equation $s+t=\frac{p-1}{2}+\ell_1$ has only one integral solution, namely $s=t=\frac{p-1}{2}$. Thus,
\beqn
-(-1)^{\frac{p-1}{2}} \sum_{w\in\mathbb{F}_p} 
A_{\frac{p-1}{2}} w^{2(p-1)} \equiv (-1)^{\frac{p-1}{2}}
\cdot \underbrace{A_{\frac{p-1}{2}}}_{=(-1)^{\frac{p-1}{2}}}=1 \mod p \,.
\eeqn
For the case $i+j+k+\ell_1 = (p-1)$, we proceed as follows
\beqn
\begin{split}
&\sum_{w\in\mathbb{F}_p} w^{\ell_1} \big( (w-a)(w-b)(w-c) \big)^{\frac{p-1}{2}} \\
&\equiv - \kern-1em \sum_{\substack{i+j+k+\ell_1 = p-1 \\ 0\le i,j,k\le\frac{p-1}{2}}}  
C^{\frac{p-1}{2}}_i C^{\frac{p-1}{2}}_j C^{\frac{p-1}{2}}_k
(-a)^{\frac{p-1}{2}-i} (-b)^{\frac{p-1}{2}-j} (-c)^{\frac{p-1}{2}-k} \\
&= - \kern-1em \sum_{i+j+k-\ell_1 = \frac{p-1}{2}} C^{\frac{p-1}{2}}_i C^{\frac{p-1}{2}}_j C^{\frac{p-1}{2}}_k
(-a)^i (-b)^j (-c)^k \\
&= - \kern-1em \sum_{i+j+k-\ell_1 = \frac{p-1}{2}}
C^{\frac{p-1}{2}}_i C^{\frac{p-1}{2}}_j C^{\frac{p-1}{2}}_k
(-1)^{\frac{p-1}{2}-\ell_1} a^i b^j c^k \mod p \,.
\end{split}
\eeqn%
By using $\ell_1=\frac{p-1}{2}-\ell \implies \ell=\frac{p-1}{2}-\ell_1$ in the second line, we obtain
\beqn
\begin{split}
|\mathcal{Y}'|_p &\equiv 1+(-1)^{\frac{p-1}{2}} 
\sum_{\ell=0}^{\frac{p-1}{2}} \Bigg( \sum_{\substack{s+t=\frac{p-1}{2}+\ell_1 \\ 0\le s,t \le\frac{p-1}{2}}}
C^{\frac{p-1}{2}}_s C^{\frac{p-1}{2}}_t C^{\frac{p-1}{2}}_{\ell_1}
2^{\frac{p-1}{2}-\ell_1} (-1)^s \Bigg) \\
&\makebox[95pt]{} \times \kern-1em \sum_{\substack{i+j+k-\ell_1 = \frac{p-1}{2} \\ 0\le i,j,k\le\frac{p-1}{2}}}  
C^{\frac{p-1}{2}}_i C^{\frac{p-1}{2}}_j C^{\frac{p-1}{2}}_k
(-1)^{\frac{p-1}{2}-\ell_1} a^i b^j c^k \\
&= 1+(-1)^{\frac{p-1}{2}} 
\sum_{\ell=0}^{\frac{p-1}{2}} \Bigg( \sum_{\substack{s+t+\ell=p-1 \\ 0\le s,t \le\frac{p-1}{2}}}
C^{\frac{p-1}{2}}_s C^{\frac{p-1}{2}}_t C^{\frac{p-1}{2}}_\ell
2^\ell (-1)^s \Bigg) \\
&\makebox[95pt]{} \times \kern-1em \sum_{i+j+k+\ell=p-1} C^{\frac{p-1}{2}}_i C^{\frac{p-1}{2}}_j C^{\frac{p-1}{2}}_k
(-1)^\ell a^i b^j c^k, \\
&= 1+(-1)^{\frac{p-1}{2}} 
\sum_{\ell=0}^{\frac{p-1}{2}} 2^\ell \kern-1em \sum_{\substack{s+t+\ell=p-1 \\ 0\le s,t \le\frac{p-1}{2}}}
\sum_{\substack{i+j+k+\ell=p-1 \\ 0\le i,j,k\le \frac{p-1}{2}}} C^{\frac{p-1}{2}}_s C^{\frac{p-1}{2}}_t C^{\frac{p-1}{2}}_\ell
C^{\frac{p-1}{2}}_i C^{\frac{p-1}{2}}_j C^{\frac{p-1}{2}}_k
 (-1)^t a^i b^j c^k. 
\end{split}
\eeqn\qedhere
\end{proof}
\small \bibliographystyle{amsplain}
\bibliography{references}{}
\end{document}